\documentclass[a4paper,11pt,reqno]{article}  
\usepackage{graphicx,color,subfigure,multirow, here}
\usepackage{mathtools}
\usepackage[utf8]{inputenc} 
\usepackage[T1]{fontenc}
\usepackage[english]{babel}
\usepackage{amsmath, amsfonts, amsthm,amssymb,here,dsfont, hyperref}
\usepackage{natbib} 
\usepackage{csquotes} 
\usepackage{algpseudocode}
\usepackage[linesnumbered,ruled,vlined]{algorithm2e}
\usepackage{comment}

\newcommand{\one}{\mathds{1}}
\newtheorem{theo}{Theorem}[section]
\newtheorem{definition}[theo]{Definition}

\newtheorem{lem}[theo]{Lemma}
\newtheorem{rem}[theo]{Remark}

\newtheorem{ass}[theo]{Assumption}

\newcommand{\E}{\mathbb{E}}
\newcommand{\R}{\mathbb{R}}

\makeatletter

\@addtoreset{equation}{section}
\makeatother
\begin{document}
\title{Active learning algorithm through the lens of rejection arguments.}
\date{\today}
\author{
Christophe Denis$^{(1)}$, Mohamed Hebiri$^{(1)}$, Boris Ndjia Njike$^{(2)}$, Xavier Siebert$^{(2)}$\\
$^{(1)}$ LAMA, Université Gustave Eiffel, France\\
$^{(2)}$ Mathematics and Operational Research
University of Mons, Belgium\\
}

\maketitle

\begin{abstract}
Active learning is a paradigm of machine learning which aims at reducing the amount of labeled data needed to train a classifier.
Its overall principle is to sequentially select the most informative data points, which amounts to determining the uncertainty of regions of the input space. 
The main challenge lies in building a procedure that is computationally efficient and that offers appealing theoretical properties; most of the current methods satisfy only one or the other.
In this paper, we use the classification with rejection in a novel way to estimate the uncertain regions.
We provide an active learning algorithm and prove its theoretical benefits under classical assumptions.
In addition to the theoretical results, numerical experiments have been carried out on synthetic and non-synthetic datasets. These experiments provide empirical evidence that the use of rejection arguments in our active learning algorithm is beneficial and allows good performance in various statistical situations.
\end{abstract}

\vspace*{0.25cm} 

\section{Introduction}
\label{sec:genFramework}
The aim of machine learning consists in designing learning models that accurately maps  a set of inputs from a space $\mathcal{X}$ called \textit{instance space} to a set of outputs $\mathcal{Y}$ called \textit{label space}. Nowadays, with the data deluge, obtaining a powerful learning model requires a lot of data from $\mathcal{X}$ to be labeled, which is time consuming in many modern applications such as speech recognition or text classification. This motivated the development of other paradigms beyond classical prediction tasks. In this paper, we focus on prediction in the binary classification setting, that is $\mathcal{Y}=\lbrace 0,1\rbrace$.
In this framework, one of the most studied techniques to deal with this specificity is the iterative supervised learning procedure called \textit{active learning} \citep{cohn1994improving,castro2008minimax,balcan2009agnostic,hanneke2011rates, locatelli2017adaptivity, locatelli2018adaptive} that aims at reducing the data labeling effort by carefully selecting which data need to be labeled. The goal of \textit{active learning} is to achieve a high rate of correct predictions while using as few labeled data as possible. One of the key principles of active learning is to identify at each step the region of the instance space where the label requests should be made, called \textit{uncertain region} in this paper, also known as \textit{disagreement region} in the active learning literature \citep{hanneke2007bound, balcan2009agnostic, dasgupta2011two}.  
Many techniques have been developed to this aim, both in parametric \citep{cohn1994improving, hanneke2007bound, balcan2009agnostic, beygelzimer2009importance, hanneke2014theory} and nonparametric setting \citep{minsker2012plug, locatelli2017adaptivity, locatelli2018adaptive}.

In this paper, we are particularly interested in the nonparametric setting, where several computational difficulties have so far hampered the practical implementation of the proposed algorithms.
For example, \citep{minsker2012plug} provides interesting theoretical results which partly motivated \cite{locatelli2017adaptivity, locatelli2018adaptive} as well as the present work, but it fails to provide a computationally efficient way to estimate the uncertain region.

To overcome these shortcomings, we present a new active learning algorithm using the paradigm called \textit{rejection}. The latter typically allows the learning models to evaluate their confidence in each prediction and to possibly abstain from labeling an instance ({\it i.e.}, "reject" this instance) when the confidence in the prediction of its label is too weak. This rejection will however be used in a novel way in this work to conveniently compute the uncertain region, as explained below.

Rejection and active learning typically differ on how they are interested in this uncertain region. In rejection, the interest in the uncertain region appears \textit{after} the design of a learning model, that rejects a test point in order to avoid a misprediction. This is very useful in some applications such as medical diagnosis where a misprediction can be dramatic. However, in active learning, the uncertain region is used \textit{during} the training process to progressively improve the model's performance by requesting labels where the classification is difficult.\\
In our algorithm, we use rejection at each step $k$ of the training process to estimate the uncertain region $A_k\subset \mathcal{X}$ based on the information gathered up to this step.
Then some points are sampled from the region $A_k$ and their labels are requested.
Based on these labeled examples, an estimator $\hat{f}_k$ is provided, that is then used to assess for each $x$ $\in$ $A_k$ the confidence in the prediction. The points where the confidence is low are rejected and are considered to form the next uncertain region $A_{k+1}$, thereby progressively reducing the part of instance space $\mathcal{X}$ on which a model remains to be constructed.
We study the rate of convergence with respect to the excess-risk of our nonparametric active learning algorithm based on histograms under classical smoothness assumptions. It turns out that combining active learning sampling together with rejection allows for optimal rates of convergence. Using numerical experiments on several datasets we also show that our active learning process can be efficiently applied to any off-the-shelf machine learning algorithm. 

The paper is organized as follows : in Section~\ref{sec:background} we provide the background notions of active learning and rejection separately, then review some recent works that proposed to combine these two notions, although in a way that differs from ours. 
Then we describe our algorithm in Section~\ref{Sec:Algo} along with the theoretical guarantees about its rate of convergence. 
Practical considerations to take into account when applying our algorithm are discussed in Section~\ref{sec:practical}. 
Numerical experiments are presented in Section~\ref{sec:experiments} and we conclude the paper along with some perspectives for future work in Section~\ref{sec:conclusion}.
The full proof of our theoretical result is relegated to the Appendix.

\section{Background}
\label{sec:background}
In this Section we review the literature related to active learning in Section~\ref{subsec:active}, and the reject option framework~\ref{subsec:reject}.
Thereafter, in Section~\ref{subsec:active+reject} we provide a review on the use of the rejection in the context of active learning.

\subsection{Active learning}
\label{subsec:active}
Given an {i.i.d.} sample $(X_1,Y_1), \ldots, (X_n,Y_n)$ from an unknown probability distribution $P$  defined on $\mathcal{X}\times\mathcal{Y}$, the classification problem consists in designing a map $g: \mathcal{X}\longrightarrow \mathcal{Y}$ from the instance space to the label space.
However, building such mapping might become a tricky task in particular situations where the labeling process of input instances are only available through time-consuming or expensive requests to a so-called oracle. In such applications, one might however have access to a huge amount of unlabeled data from the instance space.
This motivated the use of the \textit{active learning} paradigm \citep{cohn1994improving} that aims at reducing the data labeling effort by carefully selecting which data to label.

Active learning algorithms were initially designed according to somewhat heuristic principles \citep{settles1994active} without theoretical guarantees on the convergence nor on the expected gain with respect to classical "passive" learning. 
The theory of active learning has then gradually developed  \citep{cohn1994improving, freund1997selective, balcan2009agnostic, hanneke2007bound, dasgupta2007general, castro2008minimax, minsker2012plug, hanneke2015minimax, locatelli2018adaptive, locatelli2017adaptivity, kpotufe2022nuances}.

We are particularly interested in the nonparametric setting, where regularity and noise assumptions are made on the regression function. Two types of regularity assumptions are made on the regression function. The first one was introduced in the seminal work by \citep{castro2008minimax} and was also used in \citep{locatelli2018adaptive}, where it is assumed that the decision boundary $\lbrace x,\;\eta(x)=\tfrac{1}{2}\rbrace$ (where $\eta$ is the regression function) is the graph of a smooth function. The second one, which was used in \citep{minsker2012plug, locatelli2017adaptivity}, assumes that the whole regression function is smooth.
In this work, we will use similar regularity assumption as in \citep{minsker2012plug}. Besides, the noise margin assumption corresponds to the so-called \textit{Tsybakov noise condition}, and it was observed that it corresponds to the situation in which active learning can outperforms passive learning~\citep{castro2008minimax}. 
\\
In this work, we design an efficient active learning algorithm, similar to that considered in \citep{minsker2012plug}, but handling the uncertain region in an explicit and computationally tractable way using rejection.

\subsection{Classification with reject option}
\label{subsec:reject}

In the present contribution, we borrow some techniques from learning with reject option. Indeed, as detailed in Section~\ref{Sec:Algo}, a core component of our active strategy relies on the confidence we have on labels of the input instances. In contrast to the classical statistical learning framework where a label is provided for each observation $x\in \mathcal{X}$, learning with reject option is based on the idea that an observation for which the confidence on the label is not high enough should not be labeled. 
From this perspective, given a prediction function $g:\mathcal{X} \to \mathcal{Y}$, an instance $x\in \mathcal{X}$ can be either classified and the corresponding label is $g(x)$ or rejected and no label is provided for $x$ (according to the literature, the output for $x$ is $\emptyset$ or any symbol as $\oplus$ meaning reject). A classifier with reject option $\tilde{g}$ is then a measurable mapping $\tilde{g}:\mathcal{X} \to \mathcal{Y}\cup \{\oplus\}$.
Reject option has been first introduced in the classification setting in \citep{Chow57}. More recently, and since the development of {\it conformal prediction} in~\citep{Vovk99IntroCP,Vovk_Gammerman_Shafer05}, reject option has become more popular and has been brought up to date to meet the current challenges. The paper by~\citep{Herbei_Wegkamp06} proposed the first statistical analysis of a classifier based on reject option. After these pioneer works, more papers on reject options appeared (e.g., \citep{Naadeem_Zucker_Hanczar10,Grandvalet_Rakotomamonjy_Keshet_Canu09,Yuan_Wegkamp10,Lei14,cortes2016learning,Denis_Hebiri19} and references therein). They mainly differ on the way they take into account the reject option. In particular, we can distinguish three main approaches: i) use the reject option to unsure a predefined level of coverage; ii) use the reject option to unsure a pre-specified proportion of rejected data; iii) consider a loss that balances the coverage and the proportion of rejected data.
It has been established that, while there is no best strategy, controlling the coverage requests more labeled data than controlling the rejection rate, which in turn asks more (unlabeled) data that the last strategy that does the trade-off. On the other hand this last approach does not control any of the two parameters.

Reject option has also been used in different contexts, such as in regression~\citep{Vovk_Gammerman_Shafer05,denis2020regression} or algorithmic fairness~\citep{NicolasEvgeniiUAI21}. These papers show how reject option can be used to efficiently solve issues that are intrinsic to the problem.

\subsection{Active learning with reject option}
\label{subsec:active+reject}
Most active learning schemes mentioned in Section~\ref{subsec:active} attempt to find the most "informative" samples in a region close the decision boundary, called \textit{uncertain region} or \textit{disagreement region}. Some recent works have refined this idea by adding an option to abstain from labeling the points ({\it i.e.}, reject) that are considered too close to the decision boundary. \\
Although the intersection of rejection and active learning seems natural, their combination is fairly recent. Current studies can be grouped into two differents settings: the first one is focused on using reject option for improving performance guarantees of some standard active learning algorithms \citep{puchkin2021exponential,zhu2022efficient} and the second one is focused on providing a classifier which takes into account reject option \citep{shekhar2021active, shah2020online}, similarly to the standard reject option setting \citep{Herbei_Wegkamp06,Denis_Hebiri19}.\\
In the first setting, \citep{puchkin2021exponential} considered the parametric framework, particularly the model misspecification. That is, given a class of classifiers $\mathcal{F}$ (which possibly do not contain the Bayes classifier), the aim is to find an estimator $\hat{f}$ which achieves minimum excess error of classification.  By using the reject option,  \citep{puchkin2021exponential} proved that exponential savings in the number of label requests are possible in model misspecification under Massart noise assumption \citep{Massart_Nedelec06}. Their algorithm is related to the disagreement-based approach \citep{hanneke2007bound, balcan2009agnostic} and outputs an improper classifier $\hat{f}$, that is $\hat{f}$ $\notin$ $\mathcal{F}$ possibly. The work of \citep{puchkin2021exponential} was extended by \citep{zhu2022efficient} which provides a more efficient active learning algorithm that overcomes the difficulty of computing the uncertain region.
In \citep{zhu2022efficient}, the authors build a classifier based on the rejection rule with exponential saving in labels, for which they establish risk bounds in a general parametric setting. At each trial, the classifier does not label points for which the doubt is substantial. This decision of abstaining from classifying a point is taken by considering a set of "good" classifiers among a parametric class of functions. In particular, a point is rejected if all "good" classifiers consider it as a difficult point, that is, the corresponding score is within the interval $[1/2 - \gamma , 1/2 + \gamma ]$, where $\gamma$ is a (small) positive real value. 
However an analysis of this algorithm sheds light on three arguments. First, the score at point $x$ should be evaluated for all "good" functions in the class. Second, tuning the parameter $\gamma$ is not discussed and it might be tricky. Finally, the empirical performance of the proposed algorithm is not considered in the paper.\\
In the second setting, \citep{shekhar2021active}, considered the nonparametric framework under some smoothness and margin noise assumptions. The authors designed an active learning algorithm which outputs a classifier that takes into account the reject option in a standard way as in \citep{Denis_Hebiri19} by deciding not to label the instances which are located near to the decision boundary.
In particular, the final outputted algorithm is a classifier with reject option. In their framework, they derived rates of convergence for an excess-risk dedicated to the reject option framework and showed that these rates are better to those obtained by the passive learning counterpart~\citep{Denis_Hebiri19}. 
However it is not obvious in this setting to obtain computationally tractable algorithms, among others because the hypothesis class needs to be restricted. 
In contrast, in the present paper, we focus on the classical active problem and derive rates of convergence for this problem, along with a practical implementation of the algorithm.



\subsection{Contributions} 
The recent works mentioned in Section~\ref{subsec:active+reject} \citep{puchkin2021exponential, shekhar2021active, zhu2022efficient} provide interesting theoretical contributions showing the interest of combining active learning and reject option. 
However the practical implementation of the related algorithms is not straightforward, notably because it is computationally difficult to estimate the uncertain region.

In this work, we use a peculiar combination of the rejection and active learning to propose an active learning which is easy to compute in practice. More precisely, our contributions are threefold:  
\begin{itemize}
    \item We transform the typical classification with reject option framework (from Sections~\ref{subsec:reject} and~\ref{subsec:active+reject}) to estimate the so-called uncertain region in a novel way.
    Not only does this methodology provide 
    a computationally efficient algorithm for active learning, but it also can be remarkably applied to any off-the-shelf machine learning algorithm.
    This is a twofold major improvement over  \citep{minsker2012plug}. 
    \item Beyond the appealing numerical properties of our procedure, we show that it achieves optimal rates of convergence for the misclassification risk and the active sampling under classical assumptions in this setting. 
    \item We illustrate the benefit of our method in synthetic and real datasets.
\end{itemize}

\section{Active learning algorithm with rejection}
\label{Sec:Algo}
In this section, after introducing some general notations and definitions, we present our algorithm in a somewhat informal way, and then provide the theoretical guarantees  under some classical assumptions. 

\subsection{Notations and definitions}
Throughout this paper $\mathcal{X}$ denotes the instance space and $\mathcal{Y} = \{0,1\}$ is the label space. 
Let $P$ be the joint distribution of $(X,Y)$. We denote by $\Pi$ the marginal probability  over the instance space and by $\eta(x)=P(Y=1\vert X=x)$ the regression function. The performance of a classification rule $g: \mathcal{X} \mapsto \{0,1\}$ is measured through the misclassification risk $R(g) = P\left(g(X) \neq Y\right)$.
With this notation, the Bayes optimal rules that minimises the risk $R$ over all measurable classification rules \citep{lugosi2002pattern} is given by $g^*(x) = \one_{\{\eta(x) \geq 1/2\}}$ and we have: 
$$
R(g^*)=1-\mathbb{E}_{\Pi}(f^*(X)) \enspace ,
$$
where $f^*(\cdot) = \max(\eta(\cdot),1-\eta(\cdot)) $ is called $\textit{score function}.$ For any classification rule $g$, the excess risk is given by
\begin{equation}
\label{eq:excess-risk}
R(g) - R(g^*) = 2\mathbb{E} \left[\left\lvert\eta(X) - \frac{1}{2} \right\rvert \, \one_{ \{ g(X) \neq g^*(X) \} }\right] \enspace .
\end{equation}

In this work, we consider the following active sampling scheme.
For each $A \subset \mathcal{X}$, and $M \geq 1$, we can sample $(X_i,Y_i)_{1 \leq i \leq M}$ {i.i.d.} random variables such that
\begin{enumerate}
    \item for all $i = 1, \ldots, M$, $X_i$ is distributed according to $\Pi(.|A)$;
    \item conditional on $X_i$, the random variable $Y_i$ is distributed according to a Bernoulli random variable with parameter $\eta(X_i)$.
\end{enumerate}
As is commonly done in the active learning setting, we assume that the marginal distribution of $X$ is known~\citep{minsker2012plug, locatelli2017adaptivity}. In the next paragraph, we describe our active algorithm for classification. 
As important tools that nicely merge the active sampling and the use of the rejection, we will pay a particular attention to the definition of the uncertain region and the rejection rate.

\subsection{Overall description of the algorithm} 
\label{subsec:descrAlgo}

With a fixed number of label requests $N$ (called the budget), our overall objective is to provide an active learning algorithm which outputs a classifier that performs better than its passive counterpart. 
The framework that we consider (Algorithm \ref{alg:active learning}) is inspired from that developed in \citep{minsker2012plug}, in which we incorporate rejection to estimate the uncertain region.

In the following, let $(\varepsilon_k)_{k \geq 0}$ be a sequence of positive numbers. Let $(N_k)_{k \geq 0}$ be a sequence defined such that $N_0 = \sqrt{N}$ and $N_{k+1} = \lfloor c_N N_k \rfloor $ with $c_N > 1$ (e.g., $c_N=1.2$ in Section~\ref{sec:experiments}). Furthermore, we consider $A_0 = \mathcal{X}=[0,1]^d$ the initial uncertain region, and thus $\varepsilon_0=1$.
We construct a sequence of uncertain regions $(A_k)_{k\geq 1}$ and for $k\geq 1$, an estimator $\hat{\eta}_k$  of $\eta$ on $A_k$ is provided.  

First, our algorithm performs an initialization phase:
\begin{itemize}
\item Initially, the learner requests the labels $Y$ of $N_0$ points $X_1, \ldots, X_{N_0}$ sampled in $A_{0}$ according to $\Pi_0=\Pi$. 
\item Based on the initial labeled data $\mathcal{D}_{N_0}= \{(X_1,Y_1),\ldots, (X_{N_0},Y_{N_0})\}$, an estimator $\hat{\eta}_0$ of $\eta$ on $A_0$ is computed and an initial classifier $g_{\hat{\eta}_{0}} = \one_{\{\hat{\eta}_{0} \geq 1/2\}}$ is provided.
\item An estimator of the score function $\hat{f}_0(x)=\max(\hat{\eta}_0(x), 1-\hat{\eta}_0(x))$ associated to $\hat{\eta}_0$ is computed.
\end{itemize}
Afterwards, our algorithm iterates over a finite number of steps until the label budget $N$ has been reached. Step $k\geq 1$ is described below.
\begin{itemize}
\item Based on the previous uncertain region $A_{k-1}$, a constant $\lambda_k$ is computed such that
conditional on the data
\begin{equation}
\label{eq:lambda}
 \lambda_k = \max\left\{t, \;\; \Pi\left(\hat{f}_{k-1}(X) \leq t |A_{k-1}\right) \leq \varepsilon_k\right\} \enspace,
\end{equation}
These $(\varepsilon_k)_{k \geq 0}$ define explicitly the \emph{sequence of the rejection rates}~\citep{Denis_Hebiri19}. 
\item This constant $\lambda_k$ is used to construct the \emph{current uncertain region} $A_k$ which is the set where the previous classifier $g_{\hat{\eta}_{k-1}}(\cdot) =  \one_{\{\hat{\eta}_{k-1}(\cdot) \geq 1/2\}} $ might fail and thus abstains from labeling :
$$
A_k = \{x \in A_{k-1}, \;\; \hat{f}_{k-1}(x) \leq \lambda_k\}\enspace,$$
where $\hat{f}_{k-1}(x)=\max(\hat{\eta}_{k-1}(x), 1-\hat{\eta}_{k-1}(x))$.
\item According to $\pi\left(.|A_k\right)$ the learner samples i.i.d. $(X_i,Y_i), i = 1, \ldots, \lfloor N_{k} \varepsilon_k \rfloor$ used to compute an estimator $\hat{\eta}_k$  of $\eta$ on $A_k$.
\item The learner updates the classifier over the whole space $\mathcal{X}$ as follows
$$
\hat{\eta} = \sum_{j = 0}^{k-1} \hat{\eta}_j \one_{\{A_{j} \backslash A_{j+1}\}} + \hat{\eta}_k \one_{\{A_k\}}\enspace.
$$
\end{itemize}
After the iteration process, the resulting active classifier with rejection is defined point-wise as
\begin{equation}
\label{eq:eqActiveClassifier}
\hat{g}(x) = \one_{\{\hat{\eta}(x) \geq 1/2\}}    \enspace.
\end{equation}

\subsection{Theoretical guarantees}
\label{subsec:consistency}

This section is devoted to the theoretical properties of the proposed procedure under common assumptions which are presented in Section~\ref{subsubsec:ass}. Thereafter, we state our main result in Section~\ref{subsubsec:rates_cve} that mainly shows that our algorithm achieves an optimal rate of convergence for the excess-risk when the considered classifier is the histogram rule. 

\subsubsection{Assumptions}
\label{subsubsec:ass}

We assume that $\mathcal{X} = [0,1]^d$ and consider two assumptions that are widely considered for the study of rates convergence in the passive~\citep{Audibert_Tsybakov07, Gadat_Klein_Marteau16} or active settings~\citep{minsker2012plug, locatelli2017adaptivity}. 

\begin{ass}[Smoothness assumption]
\label{ass:smoothness}
The regression function $\eta$ is $s$-Lipschitz-continuous for some $s\geq 0$, that is, for all $x,z \in[0,1]^d$: 
$$\vert\eta(x)-\eta(z)\vert\leq s.\parallel x-z\parallel_{\infty}\enspace.$$
\end{ass}

\begin{ass}[Strong density assumption]
\label{ass:strong density}The marginal probability admits a density $p_X$ and there exist constants $\mu_{min}, \mu_{max}>0$ such that for all $x \in [0,1]^d$ with $p_X(x)>0$, we have:
$$
\mu_{min}\leq p_X(x)\leq \mu_{max} \enspace.
$$
\end{ass}

Assumption~\ref{ass:smoothness} imposes the regularity of the regression function $\eta$  while Assumption~\ref{ass:strong density} ensures in particular that the marginal distribution of $X$ admits a density which is bounded from below. Furthermore, we also assume that $f(X)$ admits a bounded density. 

\begin{ass}[Score regularity assumption]
\label{ass:boundedness}Let $f(x)=\max(\eta(x),1-\eta(x))$ be the score function. The random variable $f(X)$ admits a bounded density (bounded by $C>0$).
\end{ass}
Assumption~\ref{ass:boundedness} has two important consequences. The first one is that the cumulative distribution function $F_{f}$
of $f(X)$ is Lipschitz. The second one is that the so-called Margin assumption~\citep{Tsybakov04} is fulfilled with margin parameter $\alpha = 1$. This Margin assumption is also considered in~\citep{minsker2012plug} for the study of optimal rates of convergence in the active learning framework.

\subsubsection{Rates of convergence}
\label{subsubsec:rates_cve}

In this section, we present our main theoretical result (Theorem~\ref{theo-rate}) which highlights the performance of our algorithm.
While our methodology can handle any machine learning algorithm for the estimation of the regression function $\eta$, we provide theoretical guarantee with the histogram rule (whose definition is recalled in Definition~\ref{def:estimator}) for the estimation of the regression function at each step of the procedure described in Section~\ref{subsec:descrAlgo}, as in~\citep{minsker2012plug}. For completeness, we provide the full proof of our result in this particular case in the Appendix.

Let us denote by $\mathcal{C}_r=\lbrace R_i, \;i=1,\ldots, r^{-d}\rbrace$ a cubic partition of $[0,1]^d$ with edge length $r>0$.

\begin{definition}[Histogram rule]
\label{def:estimator}
Let $A$ be a subset of $[0,1]^d$. Consider a labeled sample $\mathcal{D}_{N_A}= \left\{ (X_1^A,Y_1),\ldots,(X_{N_A}^{A},Y_{N_A})\right\} $ of size $N_A \geq 1$, such that $X_i^A$ $(i=1,\ldots,N_A)$ is distributed according to $\Pi(.\vert A)$. 
The histogram rule on $A$ is defined as follows.
Let $R_i$ $\in$ $\mathcal{C}_r$ with $R_i\cap A\neq \emptyset$. For all $x \in R_i$, 
$$\hat{\eta}_{A,N_A,r}(x) = \frac{\Pi(A)}{\Pi(R_i)}\dfrac{1}{N_A} \sum_{j=1}^{N_{A}} Y_j \one_{\{X_j \in R_i\}}\enspace.$$
\end{definition}
It is known that in the passive framework, the histogram rule achieves optimal rates of convergence~\citep{Devroye_Gyorfi_Lugosi96}.

\begin{theo}
\label{theo-rate}
Let $N$ be the label budget, and $\delta \in \left(0,\frac{1}{2}\right)$. Let us assume that Assumptions~\ref{ass:smoothness}, ~\ref{ass:strong density}, and ~\ref{ass:boundedness} are fulfilled. 
At each step $k \geq 0$ of the algorithm presented in Section~\ref{subsec:descrAlgo}, we consider
\begin{enumerate}
    \item[i)] $\hat{\eta}_k:=\hat{\eta}_{A_k,\lfloor N_k\Pi(A_k)\rfloor,r_k}$, with $r_k=N_k^{-1/(d+2)}$,
    \item[(ii)] and define $(\varepsilon_k)_{k\geq 0}$ as $\varepsilon_0=1$, and for $k\geq 1$,
    $\varepsilon_k=\min\left(1,\log\left(\frac{N}{\delta}\right)\log(N) N_{k-1}^{-1/(2+d)}\right)$.
\end{enumerate}
Then with probability at least $1-\delta$, the resulting classifier defined in Equation\eqref{eq:eqActiveClassifier} satisfies 
\begin{equation}
\label{eq:rate} 
R(g_{\hat{\eta}})-R(g^*)\leq \widetilde{O}\left(N^{-\frac{2}{1+d}}\right),
\end{equation}
where $\widetilde{O}$ hides some constants and logarithmic factors.
\end{theo}
The above result calls for several comments. First, our active classifier $\hat{g}$ based on the histogram rule is optimal for the active sampling {\it w.r.t.} the misclassification risk up to some logarithmic factors (see~\citep{minsker2012plug} for the minimax rates, by considering Lipschitz regression function and the margin parameter equal to $1$. 
This rate is better than the classical minimax rate in passive learning under the strong density assumption which is of order $N^{-\frac{2}{2+d}}$, see for instance~\cite{Audibert_Tsybakov07}.
Second, the sequence of the rejection rates $(\varepsilon_k)_{k \geq 0}$ should be chosen in an optimal manner guided by our theoretical findings. In particular, for each $k$, the value of $\varepsilon_k$ is of the same order as an upper bound on the error \emph{w.r.t.} the $\ell_\infty$-norm of $\hat{\eta}_{k-1}$, valid with high probability.
This value of the $\varepsilon_k$ is also linked to the probability of the uncertain region in the procedure proposed by~\cite{minsker2012plug}. 
However, the major different with the latter reference is that our rejection rate is explicit and then our algorithm can be efficiently computed due to the use of rejection arguments to determine the uncertain regions. Finally, let us notify that our work can easily be extended for Hölder regression functions with parameter $\beta$. Indeed, for $\beta\geq 1$, we can consider a similar estimator as that introduced in Definition~\ref{def:estimator} with higher order histogram rule using smoothing kernel~\citep{gine2021mathematical}.

\begin{rem}
Theorem~\ref{theo-rate} is established assuming the knowledge of the marginal distribution of $X$. This is a classical assumption in active learning that helps for sampling. However, it is possible to extend our result to unknown distributions at the price of an additional unlabeled sample and then an additional factor $1/\sqrt{\text{size of the unlabeled sample}}$.
\end{rem}

In view of the above remark, we discuss the practical implementation
of our proposed algorithm in the following section.

\section{Practical considerations}
\label{sec:practical}
Some practical aspects of the procedure are discussed in Section~\ref{subsec:numAspect} and a simple numerical illustration is provided in Section~\ref{subsec:numEx}.
The full numerical experiments are presented in Section~\ref{sec:experiments}.

\subsection{Uncertain region}
\label{subsec:numAspect}

In this section, we discuss the effective computation of the uncertain regions.
Let $k \geq 1$ represent the current step $k$ of our algorithm.
We denote by $\mathcal{D}_M = \{X_1,Y_1), \ldots,(X_M,Y_M)\}$ the data that have been sampled until step $k$. The random variable $\hat{f}_{k-1}$ is the score function built at step $k-1$.

The construction of the uncertain region $A_{k}$ relies on
$\lambda_k$ which is solution of Equation~\eqref{eq:lambda}.
First of all, we randomize the score function $\hat{f}_{k-1}$ by introducing a variable $\zeta$ distributed according to a Uniform distribution on $[0,u]$ independent of $\mathcal{D}_M$ and by defining the randomized score function $\tilde{f}_{k-1}$ as
\begin{equation*}
\tilde{f}_{k-1}(X,\zeta) =  \hat{f}_{k-1}(X) + \zeta \enspace.   
\end{equation*}
Considering the randomized score $\tilde{f}_{k-1}$ instead of $\hat{f}_{k-1}$ ensures that conditionally on $\mathcal{D}_M$,
the cumulative distribution function of $\tilde{f}_{k-1}(X, \zeta)$, denoted by $F_{\tilde{f}_{k-1}}$, is continuous. Therefore, it implies that
\begin{equation*}
\tilde{\lambda}_k = \max\left\{t, \;\; \Pi\left(\tilde{f}_{k-1}(X) \leq t |A_{k-1}\right) \leq \varepsilon_k\right\}  = F_{\tilde{f}_{k-1}}^{-1}(\varepsilon_k)\enspace.
\end{equation*}
Hence, $\tilde{\lambda}_k$ is expressed simply as the $\varepsilon_k$-quantile of the c.d.f. $F_{\tilde{f}_{k-1}}$.
To preserve the statistical properties of $\hat{f}_{k-1}$, the parameter $u$ is chosen sufficiently small ({\it e.g.,} $u \rightarrow 0$). 

Note that the computation of the c.d.f.  $F_{\tilde{f}_{k-1}}$ requires the knowledge of the marginal distribution of $X$.
In practice, this distribution may be unknown. In a second step, based on a {\it unlabeled} dataset $\mathcal{D}_{M_k}^{U}=\lbrace X_{i},i=1,\ldots,M_k\rbrace$  with $X_{i}\sim \Pi(.\vert \hat{A}_{k-1})$, and $(\zeta_1, \ldots, \zeta_{M_k})$ i.i.d. copies of $\zeta$, we  consider an estimator $\hat{\lambda}_k$ of $\tilde{\lambda}_k$
defined as follows
\begin{equation*}
\hat{\lambda}_k =  \hat{F}_{\tilde{f}_{k-1}}^{-1}(\varepsilon_k),  
\end{equation*}
where conditionally on the data, $\hat{F}_{\tilde{f}_{k-1}}$ is the empirical c.d.f. of the random variable $\tilde{f}_{k-1}(X,\zeta)$:
\begin{equation*}
\hat{F}_{\hat{f}_k}(t) = \dfrac{1}{M_k} \sum_{i=1}^{M_k} \one_{\{\tilde{f}_k(X_i,\zeta_i) \leq t\}}\enspace.
\end{equation*}
Furthermore, the unlabeled set $\mathcal{D}_{M_k}^{U}$ is assumed to be independent of $\mathcal{D}_M$, and since it remains unlabeled, it does not contribute to the budget.

Formally, the uncertain region 
$A_k$ is then defined as follows
\begin{equation*}
A_k = \left\{(x,\zeta) \in \mathcal{X} \times [0,u], \;\; \tilde{f}_{k-1}(x, \zeta) \leq \hat{\lambda}_k\right\}  \enspace.  
\end{equation*}
Therefore, $X_{M+1} \sim \Pi\left(.|A_k\right)$, is sampled from $\Pi$ such that $\tilde{f}_{k-1}(X_{M+1},\zeta) \leq \hat{\lambda}_k$ with $\zeta$ distributed according to $\mathcal{U}_{[0,u]}$.

\begin{algorithm}[h!]
\label{alg:active learning}
\KwIn{label budget $N$}
\textbf{Initialization}\\
The uncertain region $\hat{A}_0=[0,1]^d$\\
$N_0=\lfloor\sqrt{N}\rfloor$\\
$k=1$\\
$B=N_0$\\
$\varepsilon_0=1$, for all $k\geq 1$, define the rejection rate $\varepsilon_k$\\

\For{$i=1$ \textbf{to} $N_0$}{Sample i.i.d $(X_{i,0},Y_{i,0})$ with $X_{i,0}\sim \Pi$}
$\mathcal{D}_{N_0}=\lbrace (X_{1,0},Y_{1,0})\ldots,(X_{N_0,0},Y_{N_0,0})\rbrace$\\
Based on $\mathcal{D}_{N_0}$, compute an estimator $\hat{\eta}_{\mathcal{D}_{N_0}}$.\\
$\hat{\eta}_{0}:=\hat{\eta}_{\mathcal{D}_{N_0}}$

 \While{$B+\lfloor N_{k} \varepsilon_k \rfloor\leq N$}{
      Sample i.i.d $D_{M_k}^{U}=\lbrace X_{i},i=1,\ldots,M_k\rbrace$ with $X_{i}\sim \Pi(.\vert \hat{A}_{k-1})$.\\
      Based on $D_{M_k}^{U}$, compute $\hat{\lambda}_k$ such that $\widehat{\mathbb{P}}(\hat{f}_{k-1}\leq \hat\lambda_k\vert \hat{A}_{k-1})=\varepsilon_k$\\
      $\hat{A}_k:=\{x \in \hat{A}_{k-1}, \;\; \hat{f}_{k-1}(x) \leq \hat{\lambda}_k\}$\\
      $N_{k}=c_N N_{k-1}$\\
      \For{$i=1$ \textbf{to} $\lfloor N_{k} \varepsilon_k \rfloor$}{Sample i.i.d $(X_{i,k},Y_{i,k})$ with $X_{i,k}\sim \Pi(.\vert \hat{A}_{k})$}
      $\mathcal{D}_{N_k}=\lbrace (X_{1,k},Y_{1,k})\ldots,(X_{\lfloor N_{k} \varepsilon_k \rfloor,k},Y_{\lfloor N_{k} \varepsilon_k \rfloor,k})\rbrace$\\
      Based on $\mathcal{D}_{N_0}$, compute an estimator $\hat{\eta}_{\mathcal{D}_{N_k}}$\\
      $\hat{\eta}_k:=\hat{\eta}_{\mathcal{D}_{N_k}}$\\
      $\hat{\eta} = \sum_{j = 0}^{k-1} \hat{\eta}_j \one_{\{\hat{A}_{j} \backslash \hat{A}_{j+1}\}} + \hat{\eta}_k \one_{\{\hat{A}_k\}}$\\
      $B=B+\lfloor N_{k} \varepsilon_k \rfloor$\\
      $k = k+1$\\}
      
\KwOut{$\hat{g}_{\hat{\eta}}(x)=\one_{\{\hat{\eta}(x)\geq 1/2\}}  \text{for all } x \in [0,1]^d$}

\caption{Active learning with rejection}  
\end{algorithm}

\subsection{Illustrative example}
\label{subsec:numEx}

For illustrative purposes, a two-dimensional dataset of $10^6$ data points was generated using a regression function $\eta(x_1,x_2)= \frac{1}{2}(1+\sin(\frac{\pi x_2}{2}))$. 
We chose the estimators $\hat{\eta_k}$ to be linear, to make the comparison with the best linear classifier ($x_2=0$) straightforward. 
The budget was set to $N=5000$, and the sequences of $N_k$ and $\varepsilon_k$ were chosen as 
$N_k = \lfloor 1.2 \, N_{k-1}\rfloor$ and $\varepsilon_k=0.95 \, \varepsilon_{k-1}$, starting with $N_0=\lfloor\sqrt{N}\rfloor$ and $\varepsilon_0=1$. The parameter $M_k$ was set to 150.
A discussion of this choice of parameters can be found in Section~\ref{sec:parameters}.

Figure~\ref{fig:line} represents the situation after the step $k=2$ of the algorithm. 
At step $k=1$ and $k=2$, $\lambda_{k}$ has been computed using \eqref{eq:lambda}, which allows to classify the points in $\hat{A}_{k-1} \setminus \hat{A}_k$ (represented in black for $k=1$ and in brown for $k=2$). 
For visualization purposes, the points remaining in $\hat{A}_2$ have been colored according to their labels ($y=1$ in green and $y=0$ in blue), even though these labels are unknown at this step of the algorithm.
The yellow points are those in $\hat{A}_2$ whose label has already been requested to the oracle.
At subsequent steps, points in $A_k$ are selected according to the rejection rates shown in the center part of Figure~\ref{fig:line}, which shows the theoretical reject rates ($\varepsilon_k$, defined in Algorithm~\ref{alg:active learning}) in blue and the experimental ones ($\hat{\varepsilon}_k$, counted as the number of points effectively rejected) in red. The latter were computed by repeating the simulations 10 times, to present the average results along with the standard deviations in grey.
As a whole, the rejection rate is well estimated with only $M_k = 150$ unlabeled samples. However, the standard deviations indicates that the rejection rate is harder to control towards the end of the algorithm, because less points are available to estimate $\varepsilon_k$.

\begin{figure}[ht!]
\centering
\includegraphics[width=.3\textwidth]{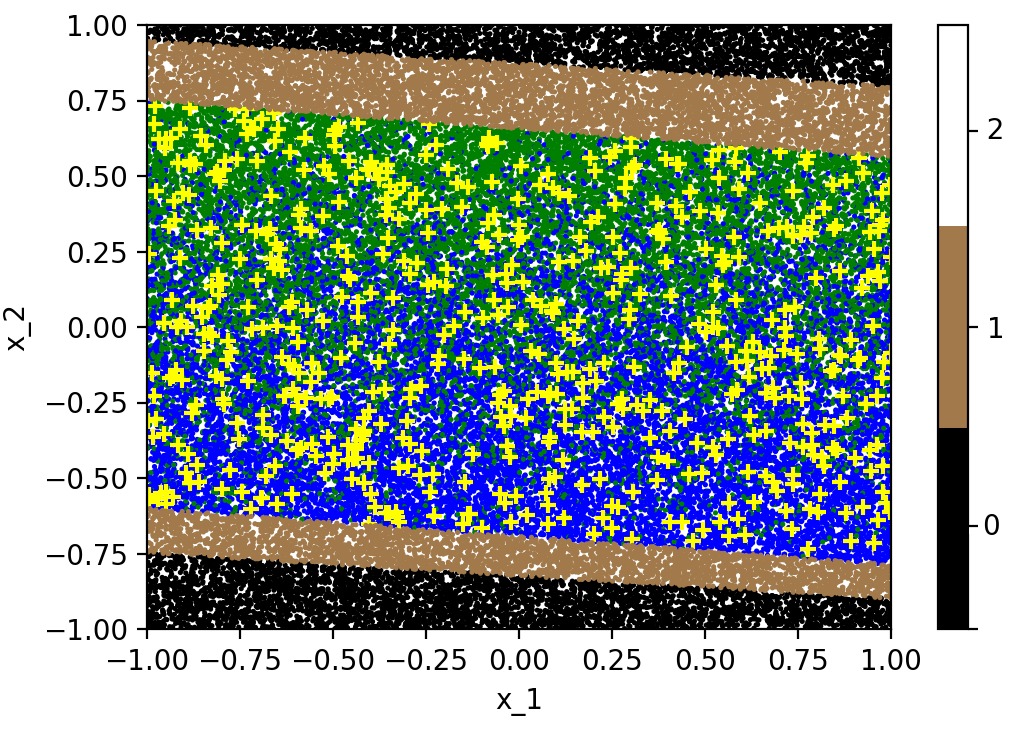}
\includegraphics[width=.3\textwidth]{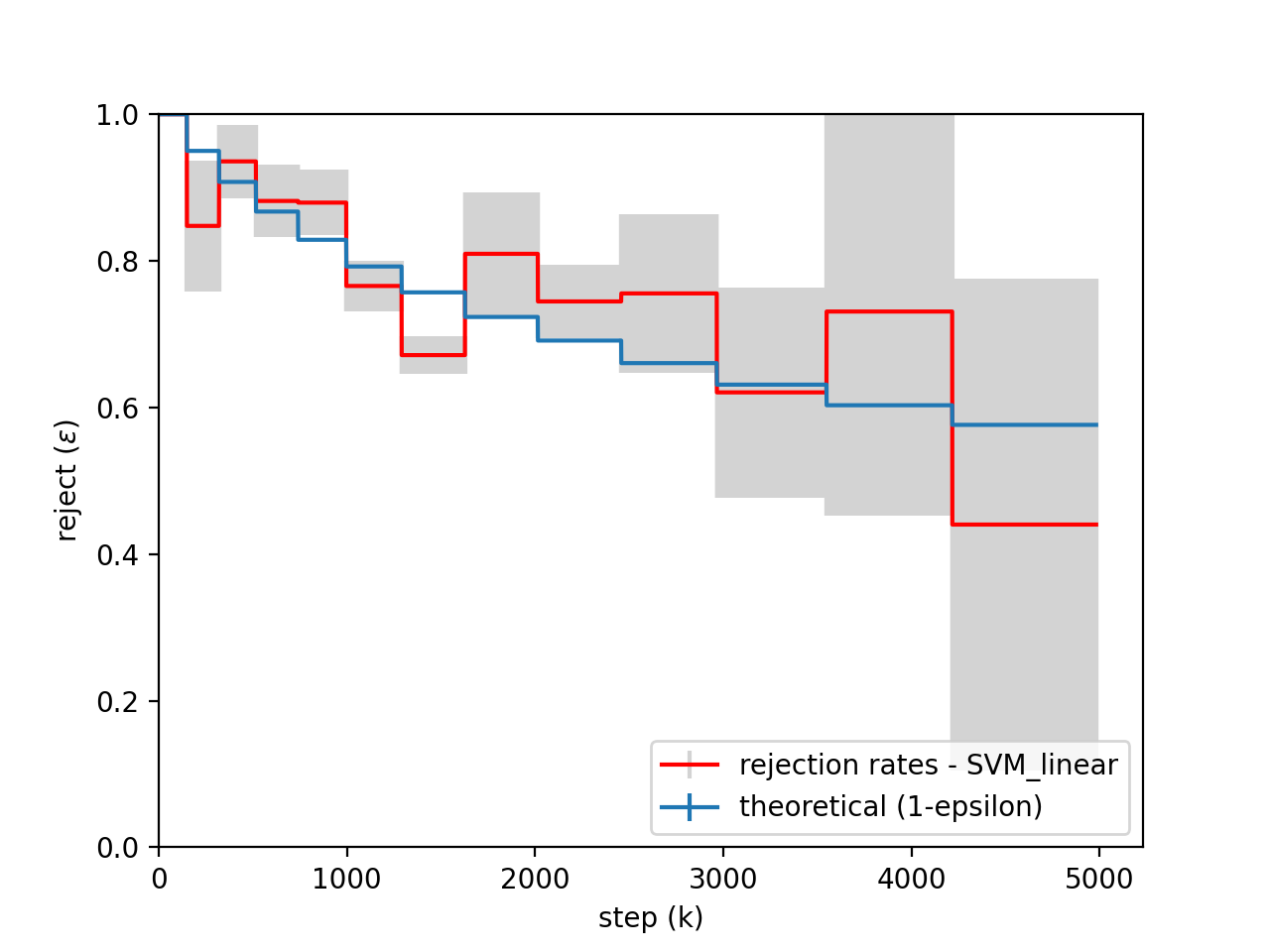}
\includegraphics[width=.3\textwidth]{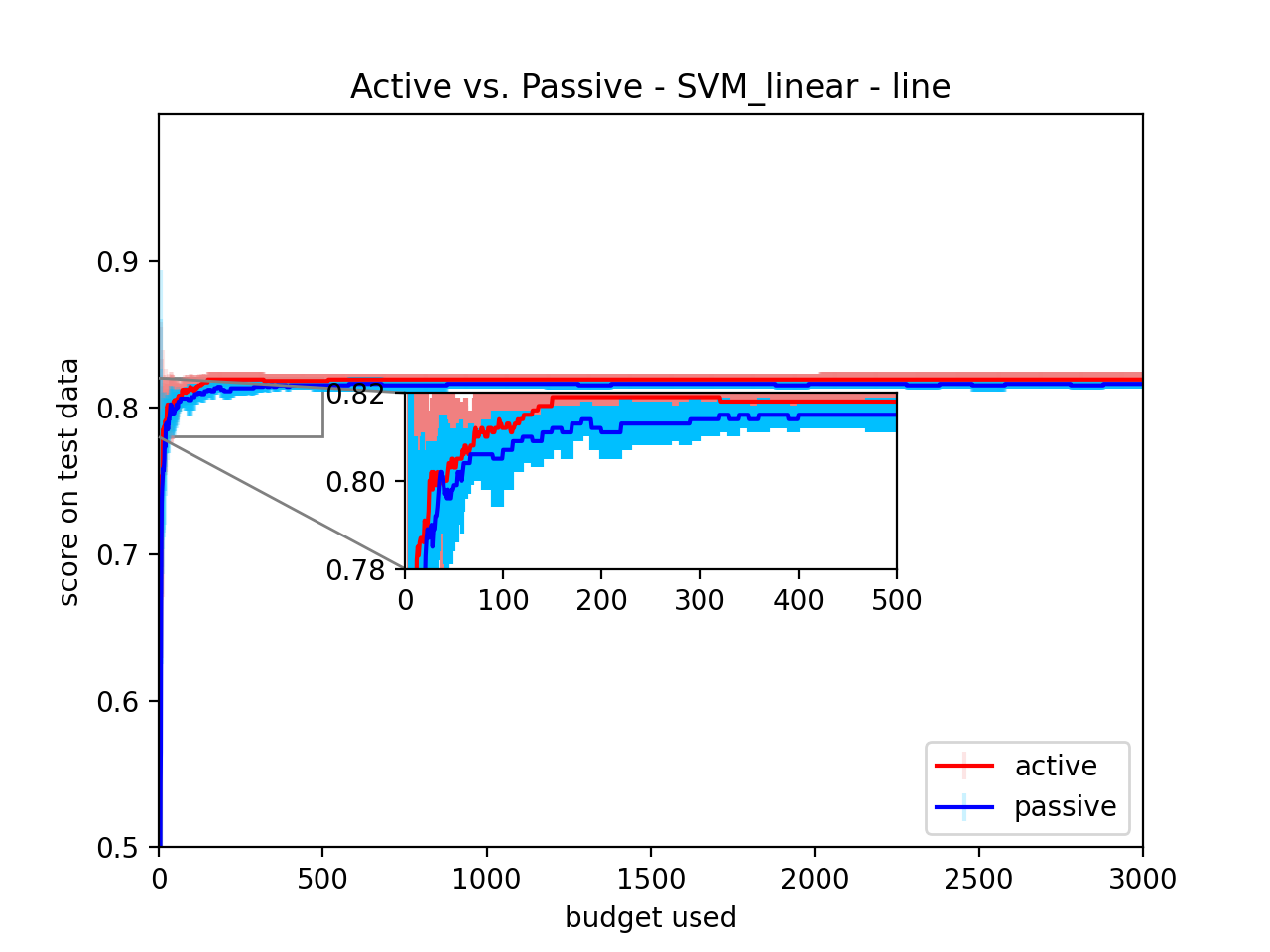}
\caption{\label{fig:line}
Left: Illustrative dataset after the step $k=2$ of the algorithm. The points in black belong to $\hat{A}_0 \setminus \hat{A}_1$ and the brown ones to $\hat{A}_1 \setminus \hat{A}_2$. In $\hat{A}_2$ are the yellow points whose label have been requested to the oracle and the remaining points in green and blue correspond to $y=1$ and $y=0$, respectively. Center: theoretical ($\varepsilon_k$, blue) and experimental ($\hat{\varepsilon}_k$, red with error bars in grey) rejection rates. Right: active vs. passive learning curves.}
\end{figure}

The resulting learning curves for passive and active procedures are represented on the right of Figure~\ref{fig:line}. 
As expected with this simplistic illustrative dataset, using active learning does not provide a substantial advantage in the long run (test precision = $0.817 \pm 0.005$ for active; $0.816 \pm 0.003$ for passive), because the optimal classifier is relatively easy to find in passive learning, even with noisy data.
However, the right panel of Figure~\ref{fig:line} shows that for a given small budget (e.g., $N < 500$), active learning converges faster than passive learning. This will be further examined in Section~\ref{sec:experiments}.

\section{Numerical experiments}
\label{sec:experiments}

\subsection{Parameters choice and sampling strategy}
\label{sec:parameters}
This Section discusses some aspects of the practical implementation of our algorithm. 
\paragraph{Parameters choice}
To perform numerical experiments, a few parameters of our model have to be set. 
First, the sequence of rejection rates was defined such that $\varepsilon_{k+1} = c_\varepsilon \varepsilon_k$, with $\varepsilon_0=1$ and $c_\varepsilon \in ]0,1[$. 
 If $c_\varepsilon$ is small, the uncertain region $\hat{A}_k$ will be small, which corresponds to an "aggressive" strategy where many points are considered to be correctly classified at each step. Conversely, if $c_\varepsilon$ is large, the strategy will be more "conservative".
Second, the constant $c_N$ defines the sequence $N_k$ as $N_k=\lfloor c_N N_{k-1} \rfloor$ and thus the number of points asked to the oracle at step $k$ ($\lfloor N_{k} \hat{\varepsilon}_k \rfloor$ on line 17 of Algorithm~\ref{alg:active learning}).
If $c_N$ is large, the algorithm will use many points at each step, thereby consuming the budget faster. A larger budget therefore allows a larger $c_N$.
Third, the number of points to build the initial classifier is theoretically set to $N_0=\lfloor\sqrt{N}\rfloor$. 
In practice, this number can be increased to get a better estimate of $\hat{\eta}_0$. Using a larger $N_0$ will however consume the budget faster. 
Third, $M_k$ unlabeled data points in $D_{M_k}^{U}$ are used at each step to estimate $\hat{\lambda}_k$. 
If $M_k$ is large, the estimation of $\hat{\lambda}_k$ will be more accurate. As these $M_k$ points remain unlabeled, they do not contribute to the budget, and $M_k$ could in principle be large. The only restriction is that at each step $k$ these (unlabeled) points have to be sampled independently of the (labeled) points asked to the oracle, it indirectly limits the number of points available to the oracle. 
Several experiments (results not shown) indicate that $M_k\ge100$ provides a reasonable estimate of $\hat{\lambda}_k$.
Finally, the parameter $u$ in Section~\ref{subsec:numAspect} has been set to $10^{-5}$. Its precise value does not affect much the results, as long as it remains close to 0.

Unless otherwise stated, our numerical experiments were performed using a "conservative approach, with the parameters discussed above set to $c_\varepsilon=0.95$, $c_N=1.2$, $N_0=2 \lfloor\sqrt{N}\rfloor$ and $M_k=150$.

\paragraph{Sampling strategy}
We designed a sampling strategy that re-uses points whenever possible, using two recycling procedures explained below.
This is not so important in our numerical experiments with synthetic data (Section~\ref{sec:synthetic}), where $10^5$ data points are used to mimic the theoretical situation with an "infinite" pool of data. However it can become crucial in practical applications with limited labeled data, as in the non-synthetic datasets used in Section~\ref{sec:non-synthetic}.

The first recycling procedure is that the unlabeled points from step $k-1$ will be re-used at step $k$. This does not invalidate our theory just because of the additive form of the risk over cells $A_k$. Indeed, our trained estimator has the form $\hat{g}(\cdot) = \sum_{k}\hat{g}_k(\cdot) \one_{A_k}(\cdot)$ and then its overall risk $R(\hat{g})$ can be decomposed on the different regions $A_k$ (by conditioning on the data used to approximate the region from the previous iteration). 

The second recycling procedure is that the data already labeled by the oracle at previous iterations (up to $k-1$ included) are reused to train $\hat{\eta}_k$, as long as they belong to the region $A_k$.
 A similar procedure was used in \citep{urner2013plal}. This allows to improve the estimation of $\hat{\eta}_k$ and to limit the budget consumption. 
 This sampling strategy is permitted because of the expression of the estimator and the decomposition of the risk as noted above.
It is particularly useful in practical applications where the total amount of labeled data is limited.

\subsection{Synthetic datasets}
\label{sec:synthetic}
\paragraph{Setting}
These numerical experiments were performed using $10^5$ data points with a budget of $N=5000$. The accuracy was tested on an independent test set of 5000 points, that were never used at any step in the algorithm. The parameters are set according to Section~\ref{sec:parameters}.

The algorithm was first challenged on three synthetic two-dimensional binary datasets (named dataset 1, 2, and 3, respectively), to study cases in which it is favorable.
Dataset 1 aims at reproducing in two dimensions a toy example used by \citep{dasgupta2011two}, where the best linear classifier is located at $x_1=-0.3$ but active learning algorithms could be misled to $x_1=0$.
Dataset 2 represents a situation where some data ($x_1<0$) are easy to classify while others ($x_1>0$) are not. 
Dataset 3 is a mixture of Gaussian distributions, whose parameters can be adjusted to create various degrees of overlap. The results presented here correspond to $\sigma=0.3$.

\begin{figure}[ht!]
\centering
\includegraphics[width=.3\textwidth]{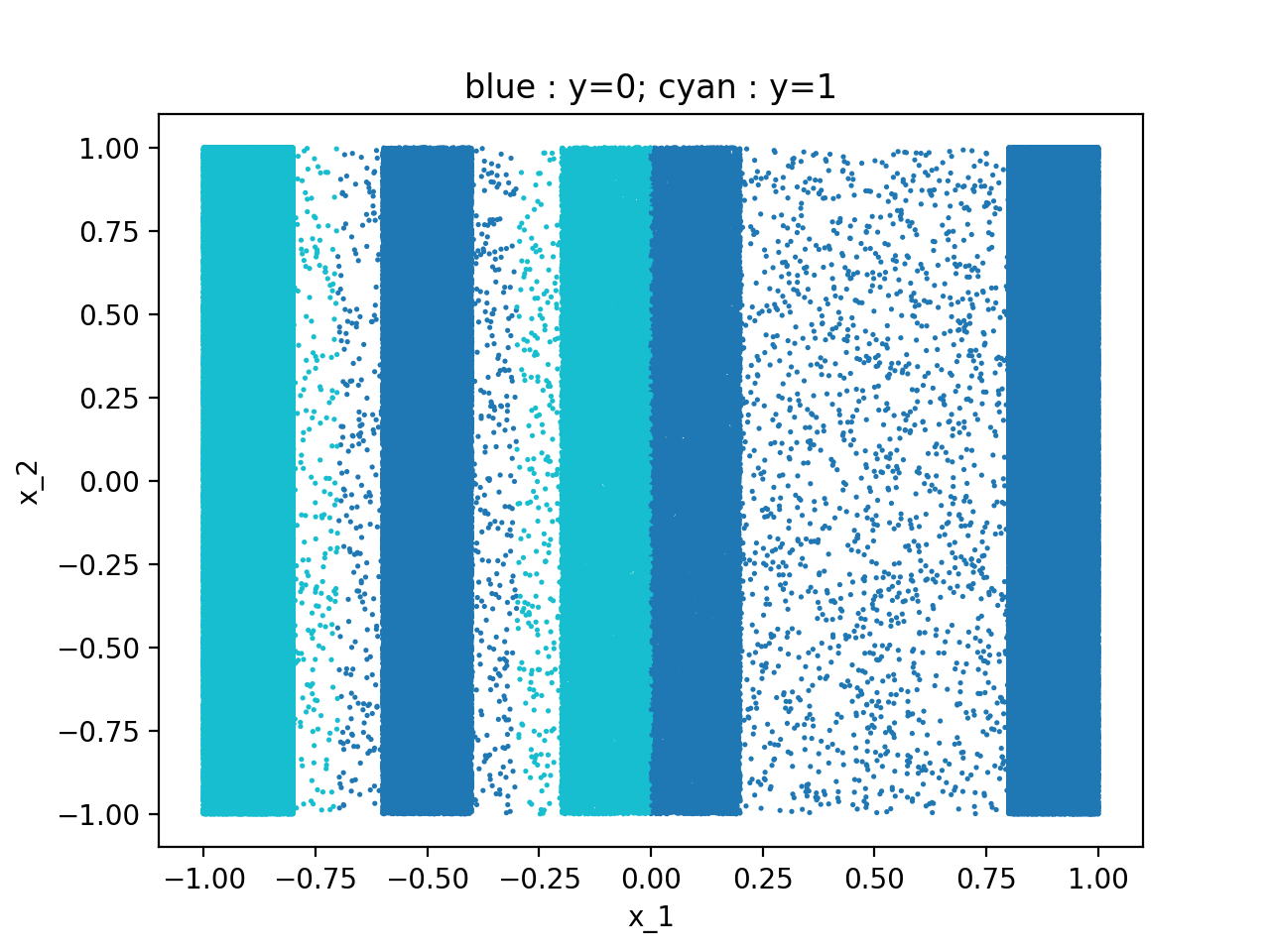}
\includegraphics[width=.3\textwidth]{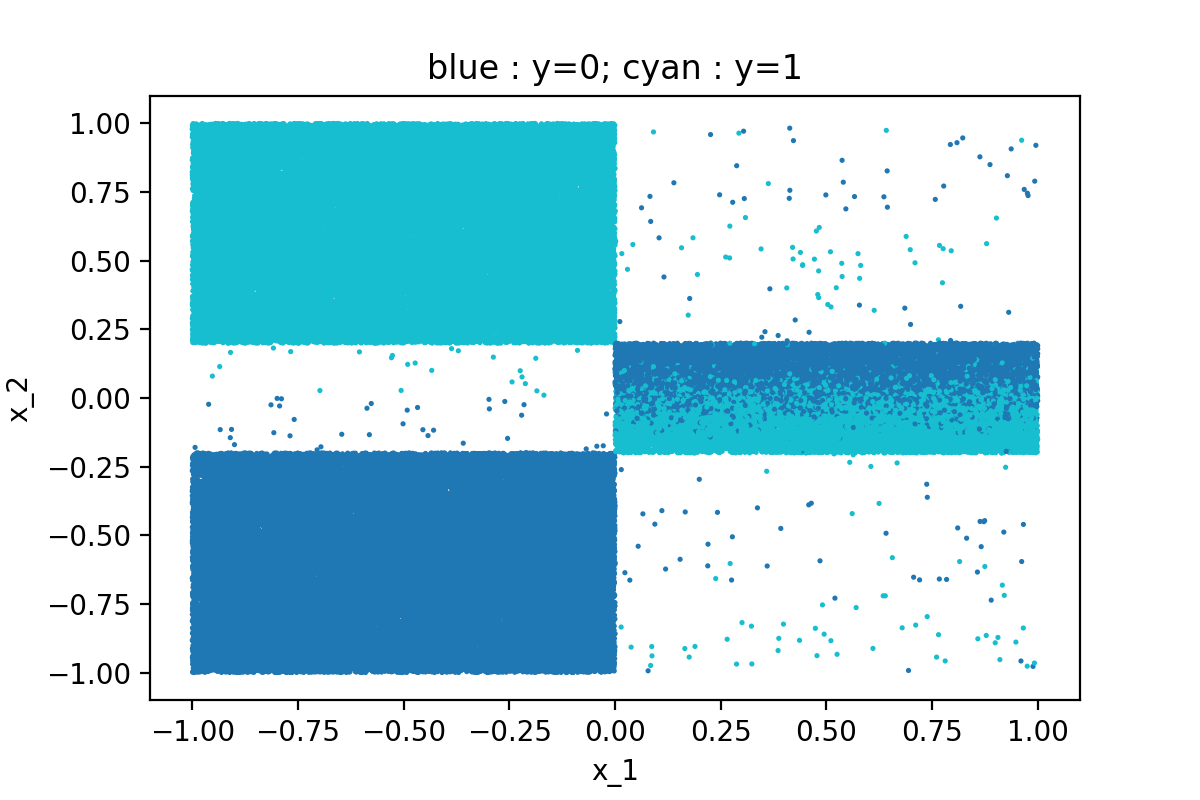}
\includegraphics[width=.3\textwidth]{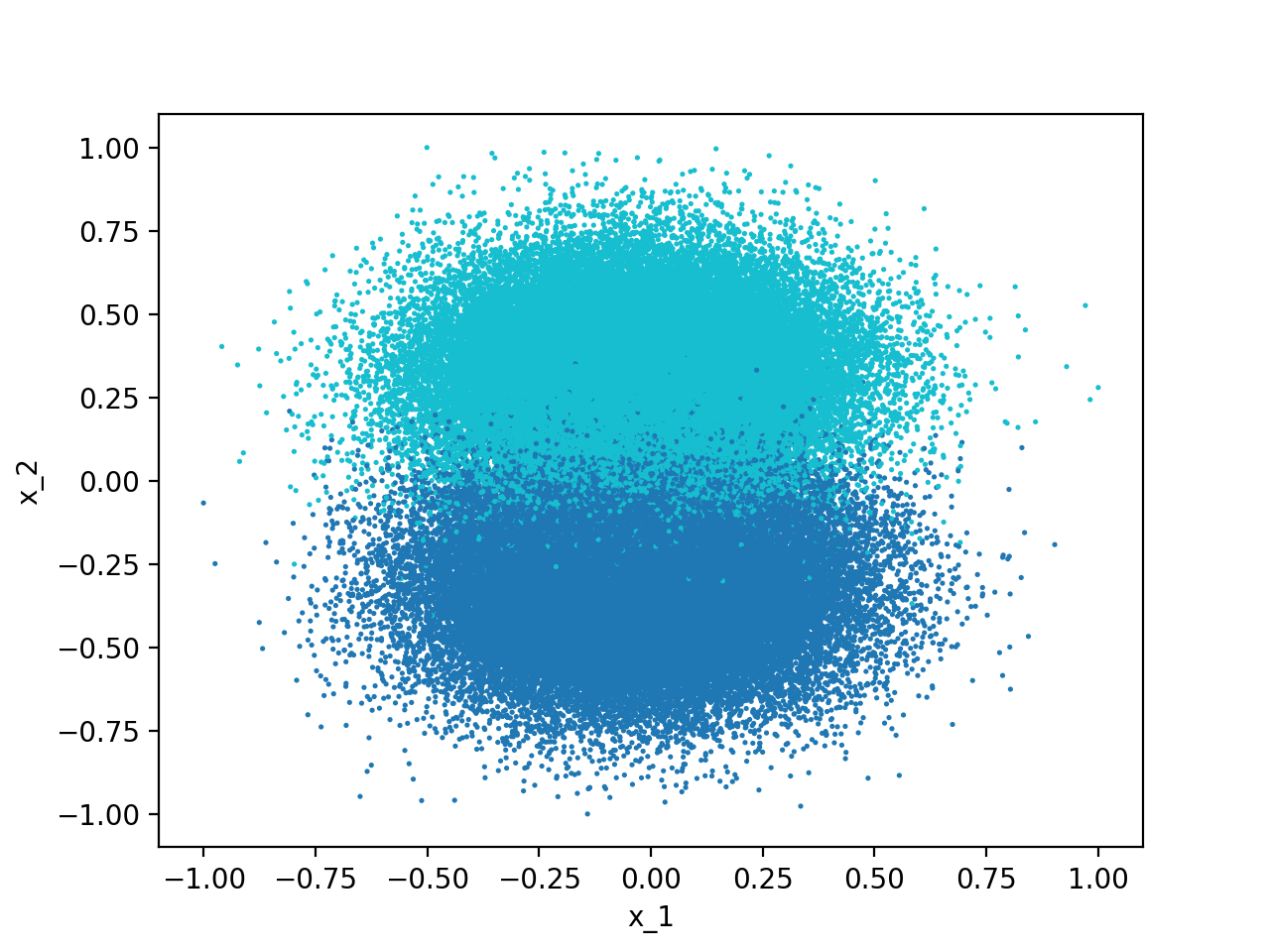} \\
\includegraphics[width=.3\textwidth]{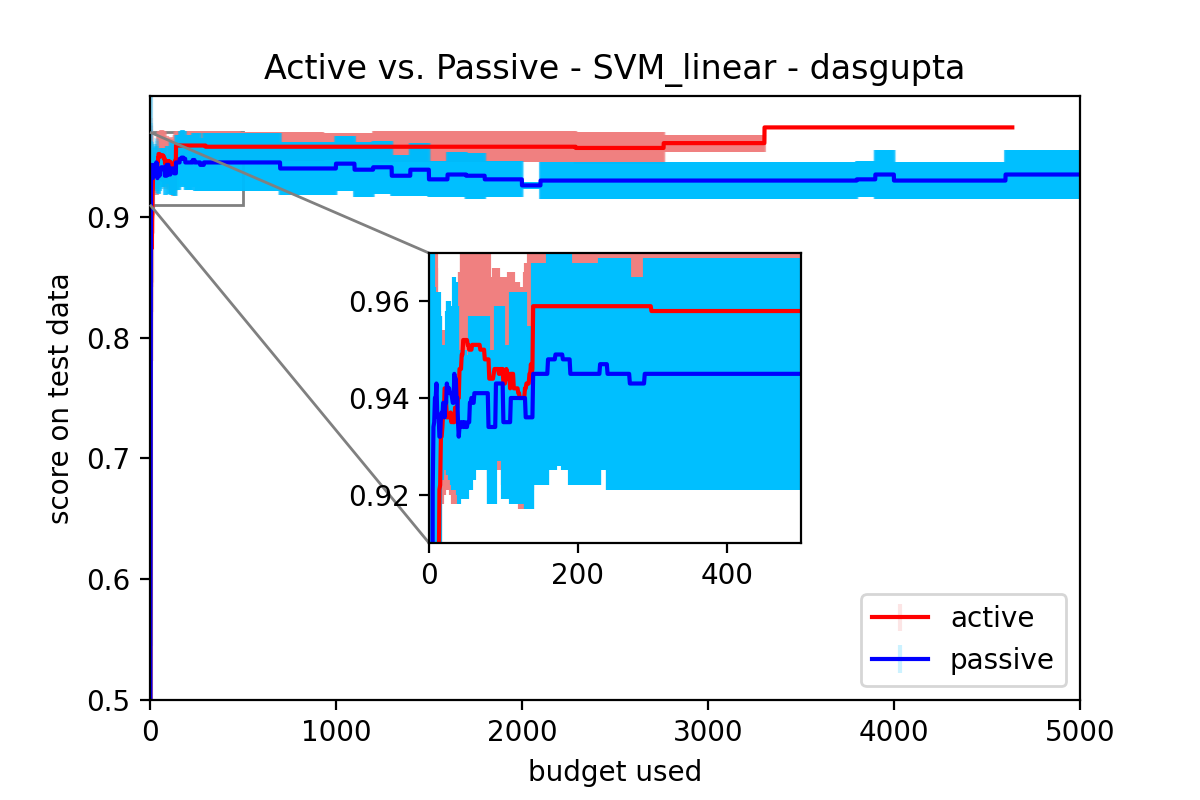}
\includegraphics[width=.3\textwidth]{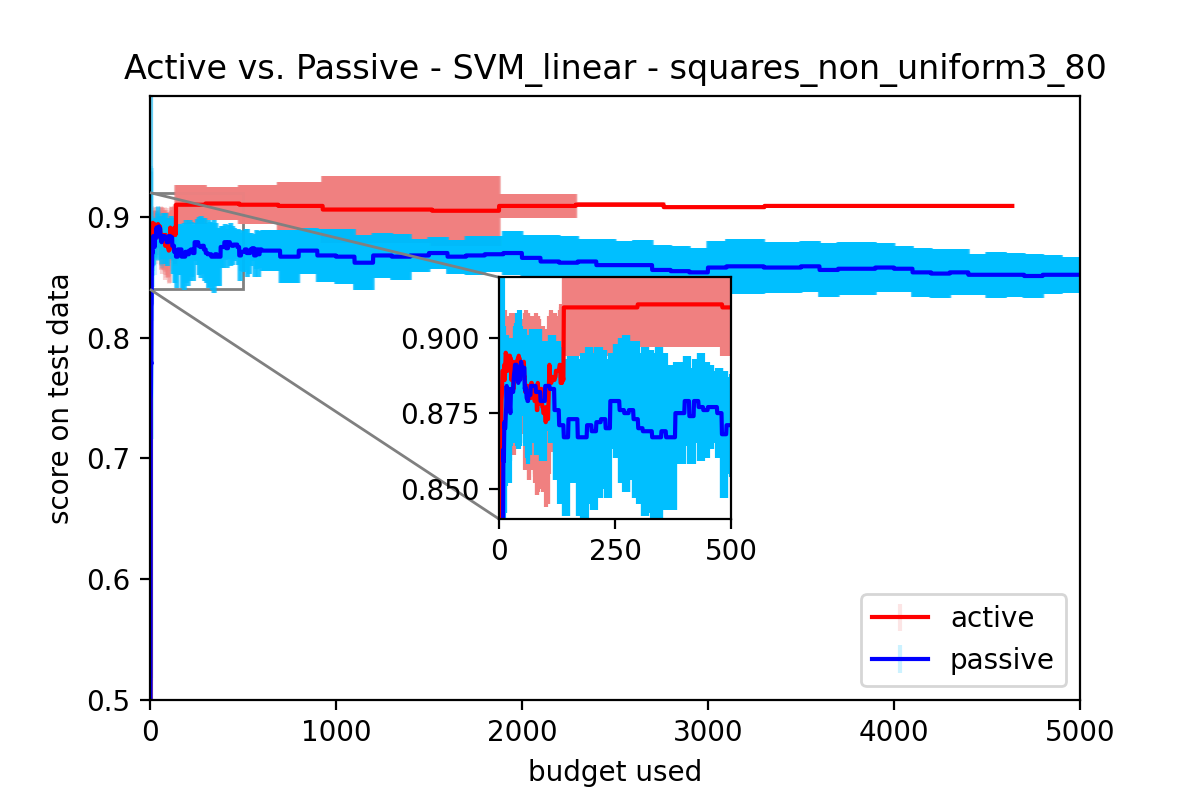}
\includegraphics[width=.3\textwidth]{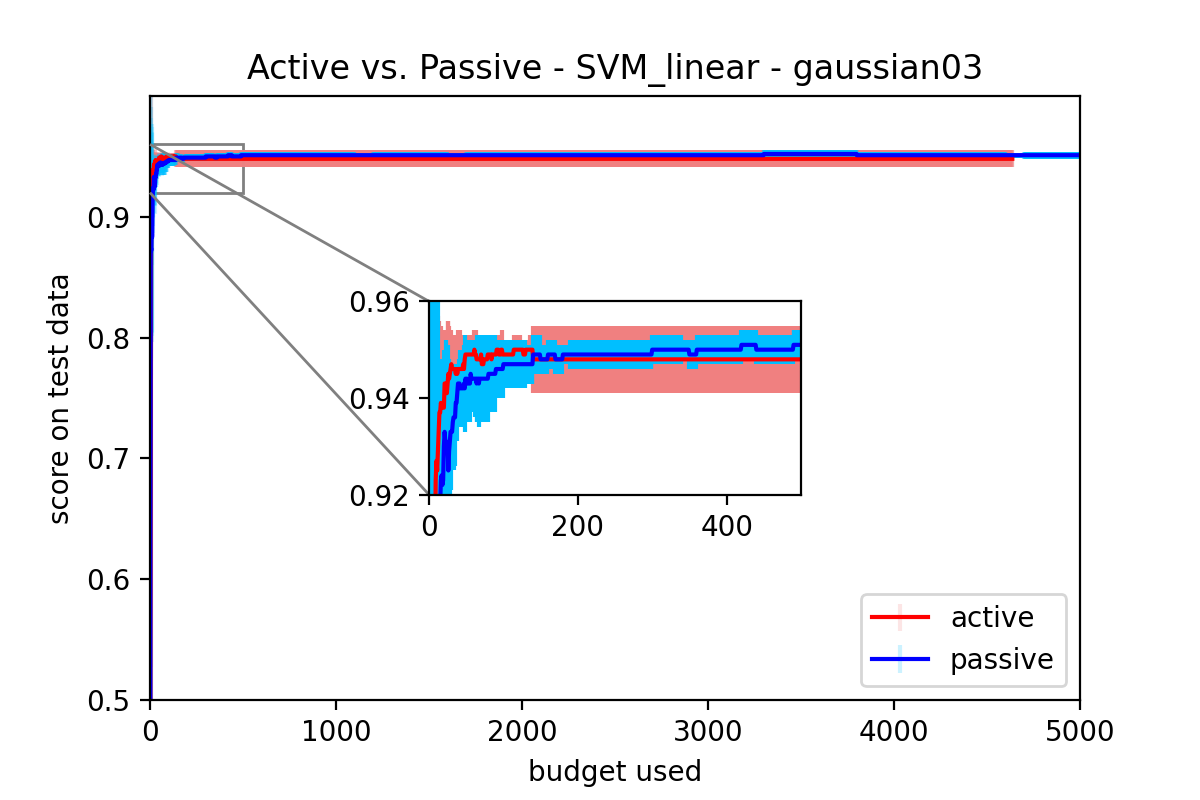}
\caption{
\label{fig:synthetic}
Top : From left to right, synthetic datasets 1, 2, and 3 used in this study with the points colored in blue or cyan depending on their class. Bottom : corresponding learning curves for active and passive linear classifiers.}
\end{figure}

The datasets are presented on Figure~\ref{fig:synthetic} as well as the corresponding learning curves for our active learning algorithm and its passive counterpart in the case of several classifiers: linear SVM, SVM with a Gaussian kernel, random forests and $k$ nearest neighbors. These classifiers are from the \texttt{scikit-learn} library \citep{scikit-learn}. Several parameters were tested, with similar results. 
The results in Table~\ref{tab:synthetic} are with the following parameters: regularization constant $C=5$ for SVM, 100 trees for random forests, $k=5$ for $k$NN. The other parameters are kept to their default value.


\begin{table}[htp]
\begin{center}
\begin{tabular}{|c|c|c|c|c|}
\hline
dataset &  classifier & budget & \multicolumn{2}{c|}{test precision} \\
\cline{4-5}
id &   & $N$ & passive & active\\
\hline
1 & SVM linear & 5000 &  0.935 $\pm$ 0.020 & \textbf{0.974 $\pm$ 0.00}  \\
 & & 200 & 0.945 $\pm$ 0.023 & \textbf{0.959 $\pm$ 0.012}  \\
& SVM rbf & 5000 & 0.975  $\pm$   0.003   & \textbf{0.996 $\pm$ 0.002}  \\
& & 200 &  0.964 $\pm$ 0.012  & \textbf{0.989 $\pm$ 0.022} \\
& random forests & 5000 &   1.000   $\pm$  0.000  &   1.000  $\pm$   0.000 \\
& & 200 &  0.989 $\pm$ 0.004  & \textbf{0.997 $\pm$ 0.008}  \\
 & kNN ($k=5$) & 5000 &  0.995 $\pm$    0.002  &  0.995  $\pm$   0.002  \\
 & & 200 & 0.956 $\pm$ 0.011 & \textbf{ 0.993 $\pm$ 0.013}  \\
\hline
2 & SVM linear & 5000 & 0.852 $\pm$ 0.015 &   \textbf{ 0.909 $\pm$ 0.000}\\
 &  & 200 & 0.871 $\pm$ 0.026  & \textbf{ 0.910 $\pm$ 0.016}  \\
& SVM rbf & 5000  &   0.966 $\pm$ 0.003  &  \textbf{0.968 $\pm$ 0.003} \\
& & 200  &  0.951 $\pm$ 0.007 &  \textbf{0.967 $\pm$ 0.005} \\
& random forests & 5000 & 0.965 $\pm$ 0.003 &    0.965 $\pm$ 0.003 \\
& & 200 & 0.957 $\pm$ 0.005 &  \textbf{0.965 $\pm$ 0.003} \\
 & kNN ($k=5$) & 5000 & 0.965 $\pm$ 0.003 &   \textbf{0.967 $\pm$ 0.003} \\
  &  & 200 &  0.950 $\pm$ 0.012  &  \textbf{0.963 $\pm$ 0.010} \\
\hline
3 & SVM linear & 5000 &  0.951   $\pm$  0.003  &  0.948  $\pm$   0.007  \\
  &  & 200 & 0.949 $\pm$ 0.003 & 0.948 $\pm$ 0.007 \\
& SVM rbf & 5000 &   \textbf{0.952  $\pm$  0.003}  & 0.943 $\pm$ 0.012  \\
& & 200 & \textbf{0.948 $\pm$ 0.003}   & 0.942 $\pm$ 0.011 \\
& random forests & 5000 &  0.944 $\pm$ 0.002  & 0.943 $\pm$ 0.006  \\
 &  & 200 & 0.942 $\pm$ 0.008   &  0.943 $\pm$ 0.007 \\
& kNN ($k=5$) & 5000 &  0.946 $\pm$ 0.004  &  0.945 $\pm$ 0.004 \\
&  & 200 &  0.944 $\pm$ 0.005  &  0.945 $\pm$ 0.003  \\
\hline
\end{tabular}
\end{center}
\caption{
\label{tab:synthetic}
Results on synthetic datasets 1, 2, and 3 for budgets of 5000 or 200, with several classifiers: linear SVM, SVM with Gaussian kernel (called SVM rbf here), random forests (with 100 trees), and $k$ nearest neighbors ($k$NN, with $k=5$)}
\end{table}

\paragraph{Results for datasets 1 and 2}

In the case of SVM linear classifiers, our active learning algorithm is clearly superior to its passive counterpart for datasets 1 and 2, either with the larger budget ($N=5000$) or with the smaller budget ($N=200$).
The situation is similar for SVM with Gaussian kernel, although it is less pronounced for dataset 2 at large budget. 
In the case of random forests and $k$NN, the difference is barely noticeable at large budget, but our algorithm is clearly superior with the smaller budget.

\paragraph{Results for dataset 3}
Dataset 3 was designed to represent an easier classification problem. In this case our active learning algorithm does not present any advantage, although it does not significantly deteriorates the results (only slightly for SVM with Gaussian kernel).

\subsection{Non-synthetic datasets}
\label{sec:non-synthetic}
Several experiments were performed with various dataset from the UCI machine learning repository.
Three "large" (more than 10000 data points) were used: \textit{skin} (245057 points in $\R^3$), \textit{fraud} (20468 points in $\R^{113}$) and \textit{EEG} (14980 points in $\R^{14}$).
For those "large" datasets a maximum budget of $N=3000$ was used.

Three "small" (less than 1000 data points) were also considered: \textit{breast} (683 points in $\R^{10}$), \textit{cleveland} (297 points in $R^{13}$), \textit{credit} (690 points in $R^{14}$).
For those "small" datasets a maximum budget of $N=500$ was used.

The results for the largest dataset (\textit{skin}) are presented as learning curves on Figure~\ref{fig:skin}. All results are summarized in Table~\ref{tab:large-non-synthetic}.

\begin{figure}[ht!]
\centering
\includegraphics[width=.4\textwidth]{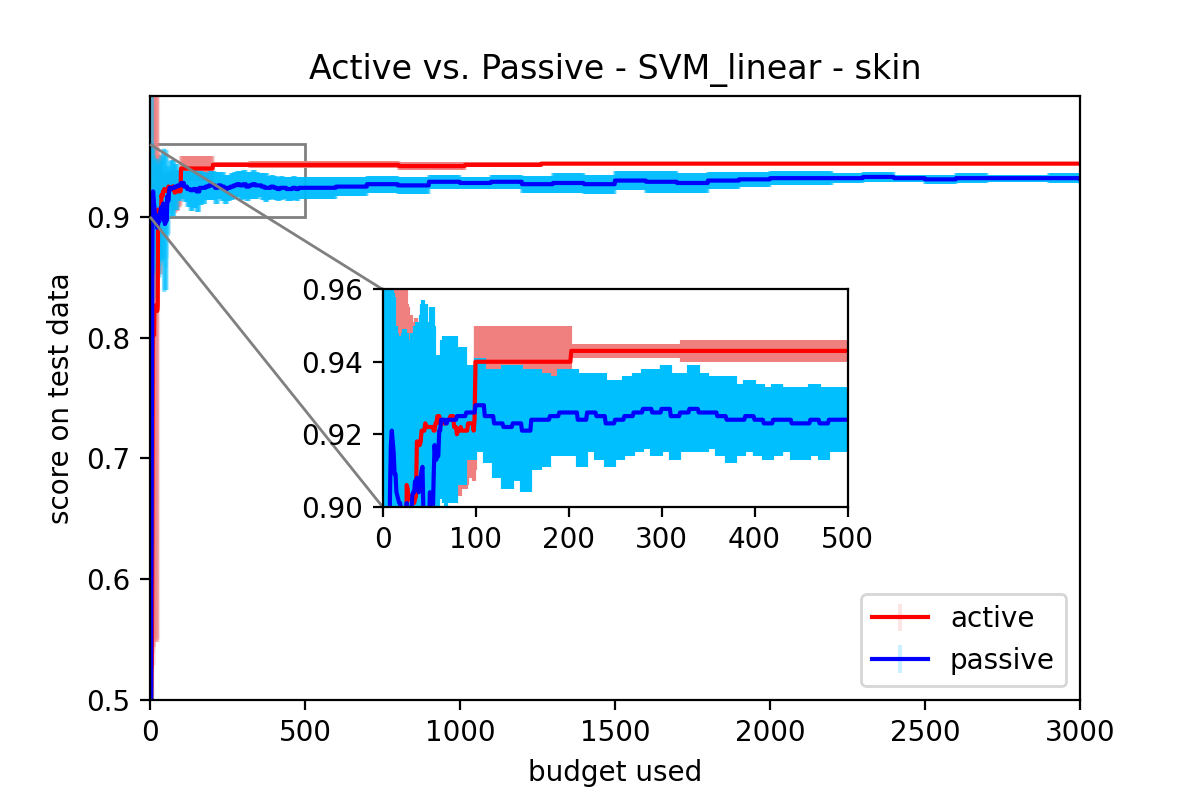}
\includegraphics[width=.4\textwidth]{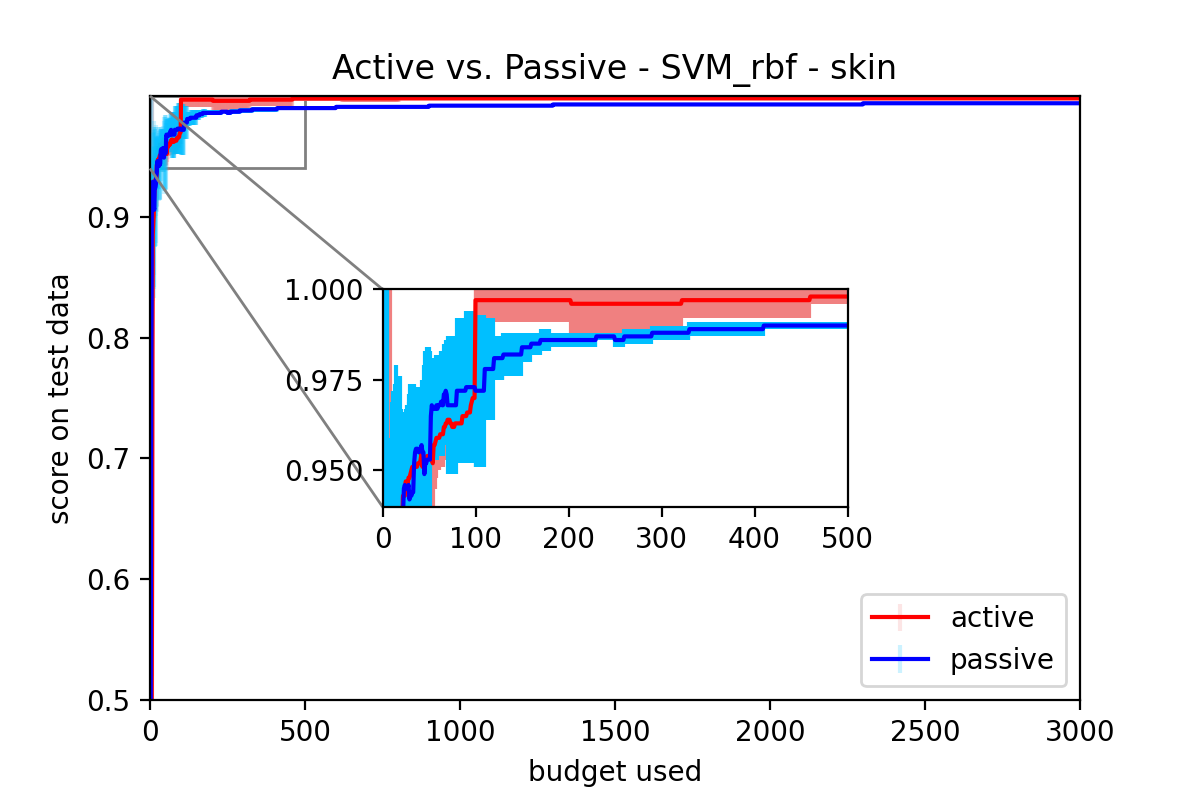}
\includegraphics[width=.4\textwidth]{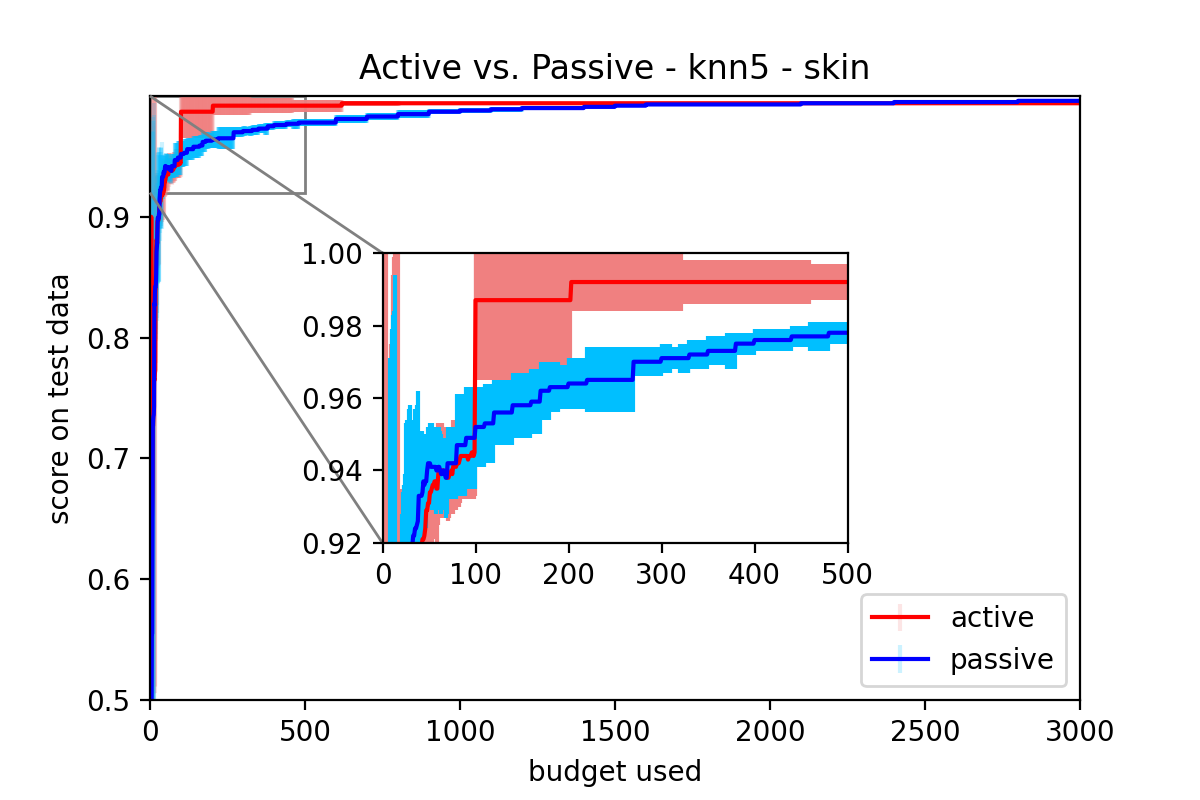}
\includegraphics[width=.4\textwidth]{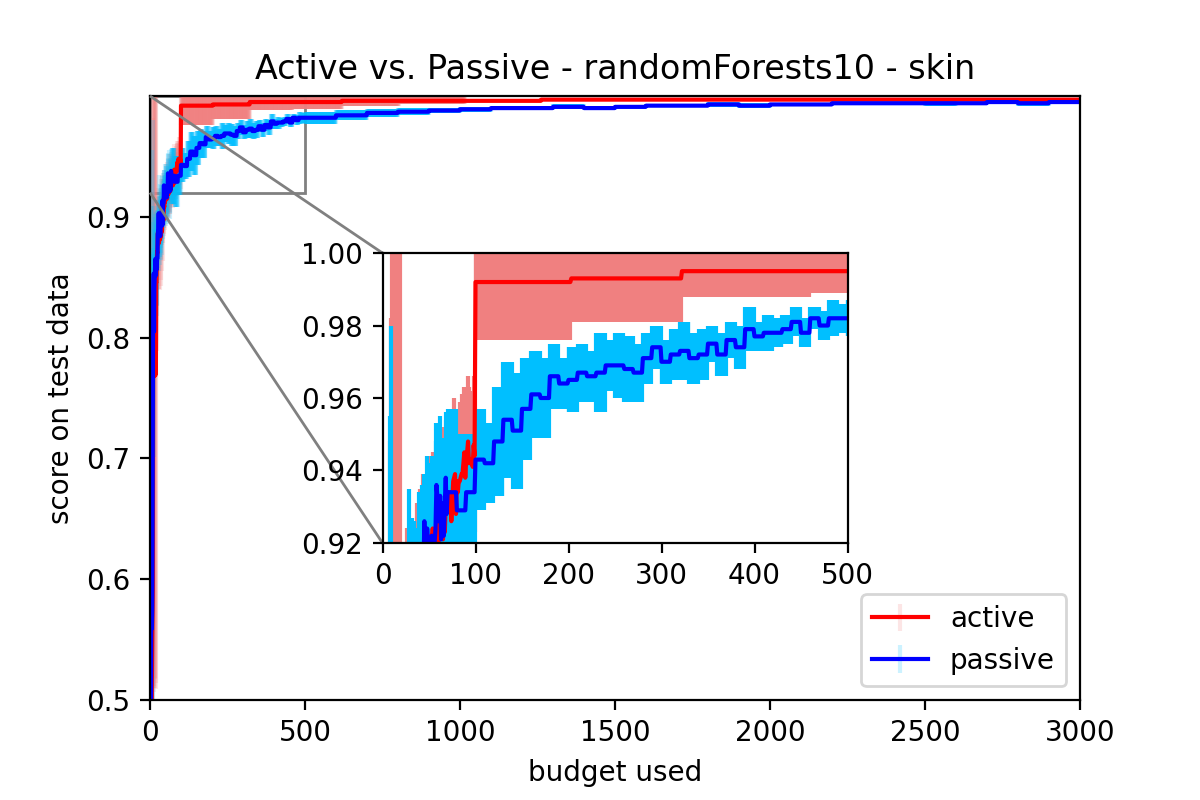}
\caption{\label{fig:skin}
Skin dataset with linear SVM, rbf SVM, kNN5, and random forests: active vs. passive learning curves.}
\end{figure}

These results indicate that for the \textit{skin} and \textit{fraud} datasets, the converged accuracy (at large budget) is superior for active learning in the case of SVM linear, but very similar for the other classifiers. This is partially due to the fact that the resulting active classifier is not linear anymore.
However, when the budget is limited to smaller values (see the inserts of Figure~\ref{fig:skin}),  the active learning procedure provides a clear advantage.

\begin{table}[htp]
\begin{center}
\begin{tabular}{|c|c|c|c|}
\hline
name & classifier & passive & active\\
\hline
skin  & SVM linear & 0.931 $\pm$ 0.004  & \textbf{0.944 $\pm$ 0.001} \\
& SVM rbf & 0.994 $\pm$ 0.000 & \textbf{0.998 $\pm$ 0.000} \\
& random forests & 0.995 $\pm$ 0.001 &  \textbf{0.997 $\pm$ 0.000} \\
& kNN ($k=5$) & \textbf{0.996 $\pm$ 0.000} & 0.994 $\pm$ 0.000 \\
\hline
fraud  & SVM linear &  0.994 $\pm$ 0.000 &  \textbf{0.999 $\pm$ 0.000} \\
& SVM rbf & 0.988 $\pm$  0.002 & \textbf{0.993 $\pm$ 0.001}  \\
& random forests & 0.991 $\pm$  0.006 & \textbf{0.998 $\pm$ 0.002} \\
 & kNN ($k=5$) & 0.946 $\pm$ 0.003 & \textbf{0.959  $\pm$ 0.002} \\
\hline
EEG & SVM linear & \textbf{0.555 $\pm$ 0.005} &  0.534 $\pm$ 0.000  \\
& SVM rbf &  0.549 $\pm$ 0.006 & \textbf{0.559 $\pm$ 0.000} \\
& random forests & 0.833 $\pm$ 0.005 & \textbf{0.877 $\pm$ 0.028} \\
 & kNN ($k=5$) & \textbf{0.763 $\pm$ 0.007} &  0.716  $\pm$ 0.009  \\
\hline
\end{tabular}
\end{center}
\caption{
\label{tab:large-non-synthetic}
Results on "large" non-synthetic datasets with several classifiers for active and passive procedures, with a budget of $N=3000$.}
\end{table}

The picture remains unchanged when we consider the "small" datasets. Indeed, most of the time the active method improves the passive one (see Table~\ref{tab:small-non-synthetic}). However, this improvement is rather limited, expect for \textit{cleveland} dataset where the use of the active algorithm is particularly beneficial.

\begin{table}[htp!]
\begin{center}
\begin{tabular}{|c|c|c|c|}
\hline
name &  classifier & passive & active\\
\hline
breast & SVM linear & 0.965  $\pm$    0.008  & \textbf{0.972  $\pm$   0.006} \\
& SVM rbf &  0.961 $\pm$   0.008 &  \textbf{0.968 $\pm$  0.011} \\
& random forests &0.968 $\pm$  0.009  &   \textbf{0.970  $\pm$  0.008}  \\
 & kNN ($k=5$) &   0.964 $\pm$  0.008 & 0.965 $\pm$ 0.011 \\
\hline
cleveland & SVM linear & \textbf{0.829 $\pm$  0.047}  & 0.821 $\pm$ 0.011  \\
& SVM rbf &   0.804 $\pm$  0.025 & \textbf{0.906 $\pm$ 0.017}  \\
& random forests &  0.778 $\pm$  0.029  & \textbf{0.879 $\pm$ 0.059}   \\
 & kNN ($k=5$) &    0.797 $\pm$ 0.038 & \textbf{0.815 $\pm$  0.014} \\
\hline
credit & SVM linear &  0.848 $\pm$ 0.023  &   0.847 $\pm$  0.020 \\
& SVM rbf &  0.862 $\pm$ 0.017  &  0.851 $\pm$   0.022  \\
& random forests & 0.845 $\pm$  0.025  & \textbf{0.853 $\pm$ 0.014}  \\
 & kNN ($k=5$) &   0.851 $\pm$  0.024  & \textbf{0.857 $\pm$  0.019} \\
\hline
\end{tabular}
\end{center}
\caption{\label{tab:small-non-synthetic}Results on three "small" non-synthetic datasets with several classifiers and a budget not to exceed 500.}
\end{table}

\subsection{Summary of the results and discussion}
The study on synthetic datasets shows that our active learning algorithm using rejection provides a clear advantage for the first two datasets, especially at low budget, but not for the third dataset.
This indicates that our algorithm is most useful in situations where the classification problem is more difficult.

In non-synthetic datasets, the active learning procedure appears to be most effective on larger datasets. The explanation is as follows. 
For small datasets (\emph{e.g.}, a few hundreds points), the number of points $N_0$ has to be chosen quite small. The estimate $\hat{\eta}_0$ is thus likely to be inaccurate, which in turn implies an inaccurate estimation of the uncertain region in the first steps and then leads to a poorly controlled algorithm. 

Interestingly, even in such small datasets, our algorithm is rarely detrimental to the final precision reached and can even be useful when the budget is extremely limited. 

\section{Conclusion and perspectives}
\label{sec:conclusion}
Recently several works have started to combine active learning and rejection arguments by abstaining to label some data within an active learning algorithm. 
This combination is very natural since active learning and rejection both focus on the most difficult data to classify. 
In this work, instead of completely abstaining to label some data, we use rejection principles in a novel way to estimate the uncertain region typically used in active learning algorithms.
We therefore propose a computationally efficient active learning algorithm that combines active learning with rejection. 
We theoretically prove the merits of our algorithm and show through several numerical experiments that it can be efficiently applied to any off-the-shelf machine learning algorithm. The benefits are more pronounced when the label budget is limited, which is promising for practical applications.

Nevertheless, in the last steps of our algorithm the uncertainty about the label of some points can become very substantial, in which case it becomes natural to completely abstain from labeling. This abstention will be included in future work combined with our use of the reject option.

\newpage
\bibliographystyle{ScandJStat}
\bibliography{activeLearningReject}

\newpage

\appendix

\section*{Appendix}

The section is devoted to the proof of  the main result.

\section{Technical result}

Let us first introduce some general notations: Let $A$ be a subset of $[0,1]^d$, and a cubic partition $\mathcal{C}_r$ as introduced in Definition \ref{def:estimator}.
For $R\in \mathcal{C}_r$,  with $R\cap A\neq \emptyset$, we introduce the regression function in $R$ as:

$$\bar{\eta}(R) = \dfrac{1}{\Pi(R)} \int_{R} \eta(z) \Pi(dz\vert A).$$ and we define $\bar{\eta}(x):=\bar{\eta}(R)$ for all $x\in R$.

Here, for each $k\geq 0$, and $r_k=N_k^{-1/(d+2)}$, we consider the estimator:
\begin{equation}
\label{eq:estimator}
\hat{\eta}_k:=\hat{\eta}_{A_k,\lfloor N_k\Pi(A_k)\rfloor,r_k},
\end{equation}
where $\hat{\eta}_{A_k,\lfloor N_k\Pi(A_k)\rfloor,r_k}$ is defined according to Definition \ref{def:estimator}, and $A_k$ is defined in algorithm \ref{alg:active learning}. importantly, defining $\hat{\eta}_k$ in this way for all $k\geq 0$ allows us to characterize the set $A_k$ in an explicit form: 
$$A_k=\bigcup_{R\,\in\,\mathcal{C}_{r_k},\;R\cap A_k\neq\emptyset} R.$$

We firstly provide a high probability bound on the estimation error:

\begin{lem}[Favorable event with high probability]~\\
\label{lem:event}
Let $L$ be defined as: 
\begin{equation}
\label{eq:last-step}
   L=\max\lbrace j\geq 1, N> \sum_{k=0}^{j}\lfloor N_k\Pi(A_k)\rfloor\rbrace.
\end{equation}
Let $k$ $\in$ $\lbrace 0,\ldots, L\rbrace$ and $E$ be the event defined by: 
\begin{equation}
E=\cap_{k=0}^L E_k,
\label{eq:event}
\end{equation}
where
\begin{equation}
\label{eq:event-k}
E_k=\left\lbrace \parallel \eta-\hat{\eta}_k\parallel_{\infty,A_k}\leq c_5 \log\left(\frac{N}{\delta}\right) N_k^{-1/(2+d)}\right\rbrace,
\end{equation}
with $\parallel\eta-\hat{\eta}_k\parallel_{\infty,A_k}:= \sup_{x\in A_{k}} |\hat{\eta}_k (x) -\eta (x) |$ and 
$c_5$ is a constant independent of $N$ and $N_k$, but dependent on $L$ and $d$. 
Under Assumptions~\ref{ass:smoothness} and \ref{ass:strong density} we have: 
$$\mathbb{P}(E)\geq 1-\delta.$$
\end{lem}

\begin{proof}~\\
Let first note that $L$ is deterministic as for all $k\geq 1$, $\Pi(A_k)=\varepsilon_k$, where $\varepsilon_k$ is stated in our algorithm. Let $k$ $\in$ $\lbrace 0,\ldots, L\rbrace$ and the corresponding estimator $\hat{\eta}_k$ (see \eqref{eq:estimator}). Let $\mathcal{C}_{r_k}$ the cubic partition considered in Definition \ref{def:estimator}, and fix $R$ $\in \mathcal{C}_{r_k}$. Let $x$ $\in$ $R$ with $R\cap A_k\neq \emptyset$.\\
Let $T_{j,k} = Y_j \one_{\{X_j \in R\}} \frac{\Pi(A_k)}{\Pi(R)}$. We observe that conditional to $A_k$,    $\mathbb{E}\left[T_{j,k}\right] = \bar{\eta}(R)$, and 
\begin{equation}    
\label{eq:proof-variable-bounds}
|T_{j,k}| \leq \frac{\Pi(A_k)}{\Pi(R)}.
\end{equation}
Furthermore
\begin{equation}
\label{eq:proof-variance-bounds}
{\rm Var}(T_{j,k}) = \dfrac{\Pi(A_k)^2}{\Pi^2(R)} {\rm Var}(Y_j \one_{\{X_j \in R\}}) \leq \dfrac{\Pi(A_k)}{\Pi^2(R)} \int_{R} \eta(z) \Pi(dz\vert A_k) \leq \frac{\Pi(A_k)}{\Pi(R)}.
\end{equation}

Hence, from Bernstein Inequality, we deduce that for $t \leq 1$,
\begin{equation*}
\mathbb{P}\left(|\hat{\eta}_k(x)-\bar{\eta}(R)| \geq t \right) \leq \exp\left(-\frac{\lfloor N_k\Pi(A_k)\rfloor t^2}{Var(T_{j,k})+\frac{t\Pi(A_k)}{3\Pi(R)}}\right)\leq \exp\left(-\lfloor N_k\Pi(A_k)\rfloor\Pi(R)t^2/\Pi(A_k)\right)
\end{equation*}
by using \eqref{eq:proof-variance-bounds}.\\
Note that for $t > 1$, the inequality is always satisfied.
Now, applying the above inequality, we deduce
\begin{equation*}
\mathbb{P}\left(|\hat{\eta}(x)-\bar{\eta}(x)| \geq t\sqrt{\frac{\Pi(A_k)}{\lfloor N_k\Pi(A_k)\rfloor\Pi(R)}} \right) \leq \exp(-t^2),
\end{equation*}
Hence choosing $t = \sqrt{\log\left(\frac{N(L+1)}{c_1\delta}\right)}$, (where $c_1$ will be defined later)
we deduce that for all $x$ $\in$ $R$, with probability at least $1-\frac{c_1\delta}{N(L+1)}$, we have
\begin{equation*}
|\hat{\eta}(x)-\bar{\eta}(x)| \leq \sqrt{\log\left(\frac{N(L+1)}{c_1\delta}\right)\frac{2}{N_k\Pi(R)}}
\end{equation*}
From the strong density assumption, we then obtain that for all $x$ $\in$ $R$, with probability at least $1-\frac{c_1\delta}{N(L+1)}$,
\begin{equation}
|\hat{\eta}(x)-\bar{\eta}(x)| \leq c_2\sqrt{\log\left(\frac{N(L+1)}{c_1\delta}\right)\frac{1}{N_kr_k^d}}.
\label{eq:variance-bound1}
\end{equation}
Where $c_2=\sqrt{\frac{2}{c_1}}$, and $c_1$ is such that $\Pi(R)\geq c_1r_k^d$ by Assumption \ref{ass:strong density}.\\
To get a result in $L_{\infty}$-norm on $A_k$, it remains to consider the union bound over all $R$ $\in$ $\mathcal{C}_{r_k}$ such that $R \cap A_k \neq \emptyset$.
$$\|\hat{\eta}-\bar{\eta}\|_{\infty,A_k} \leq \max_{R,\;R \cap A_k \neq \emptyset} \|\hat{\eta} - \bar{\eta}\|_{\infty,R}.$$

By definition, for all $k\geq 0$, the estimator $\hat{\eta}_k$ is constant on each cell $R$, in this case, we have: 

$$\Pi(A_k)=\sum_{R,\;R\cap A_k=\emptyset} \Pi(R)$$
Then, by using Assumption \ref{ass:strong density}, we have: 
$$\Pi(A_k)\geq \vert \lbrace R, \;R\cap A_k\neq \emptyset\rbrace\vert\ c_1 r_k^d.$$
As $r_k=N_k^{-1/(d+2)}$, we get for all $k$ $\in$ $\lbrace 0, \ldots, L\rbrace$,   
 
\begin{equation}
\vert \lbrace R,\; R\cap A_k\neq \emptyset\rbrace\vert \leq \dfrac{1}{c_1}\Pi(A_k)N_k^{d/(d+2)}\leq\dfrac{1}{c_1}( \Pi(A_k)N_k)\leq \frac{N}{c_1}
\label{eq:card-cells}
\end{equation}
Thus we have (conditional on $A_k$): 

\begin{align*}
    \mathbb{P}\left(\forall\,x\in A_k, |\hat{\eta}(x)-\bar{\eta}(x)| > c_2\sqrt{\frac{\log\left(\frac{N(L+1)}{c_1\delta}\right)}{N_kr_k^d}}\right) &\leq \mathbb{P}\left(\max_{R,\;R \cap A_k \neq \emptyset} \|\hat{\eta} - \bar{\eta}\|_{\infty,R}>c_2\sqrt{\frac{\log\left(\frac{N(L+1)}{c_1\delta}\right)}{N_kr_k^d}}\right)\\
    & \leq\sum_{R,\;R\cap A_k=\emptyset} \mathbb{P}\left(\|\hat{\eta} - \bar{\eta}\|_{\infty,R}>c_2\sqrt{\frac{\log\left(\frac{N(L+1)}{c_1\delta}\right)}{N_kr_k^d}}\right)\\
    & \leq\vert \lbrace R,\; R\cap A_k\neq \emptyset\rbrace\vert \frac{c_1\delta}{N(L+1)}\\
    &\leq \frac{\delta}{L+1}\quad\text{by \eqref{eq:card-cells}}
\end{align*}
Besides, Assumption \ref{ass:smoothness} leads to
\begin{equation}
\|\eta-\bar{\eta}\|_{\infty, A_k}\leq c_3 r_k,
\label{eq:bias-bounds1}
\end{equation}
where $c_3$ depends on $s$ (from Assumption \ref{ass:smoothness}) and $d$.
Thus, by combining \eqref{eq:variance-bound1}, \eqref{eq:bias-bounds1} and \eqref{eq:card-cells}, we can obtain that with probability at least $1-\frac{\delta}{L+1}$,
\begin{equation*}
\|\hat{\eta}_k-\eta\|_{\infty,A_k} \leq c_4\left( \sqrt{\log\left(\frac{N(L+1)}{c_1\delta}\right)\frac{1}{N_k r_k^d}}+r_k\right),
\end{equation*}
where $c_4=\max(c_2,c_3)$.\\
Finally, as $r_k = N_k^{-1/2+d}$, by considering the union bound over all steps, we get with probability at least $1-\delta$, 
\begin{equation}
\label{eq:union-bound-final}
\|\hat{\eta}_k-\eta\|_{\infty,A_k} \leq c_5 \log\left(\frac{N}{\delta}\right) N_k^{-1/(2+d)}\quad \text{for all}\;k\in\lbrace0,\ldots,L\rbrace
\end{equation}
where $c_5$ depends on $c_4, c_1$ and $L$.
\end{proof}

Because the constant $c_5$ in \eqref{eq:event-k} depends on $L$, we provide below a result which states that the variable $L$ defined in \eqref{eq:last-step} does not affect drastically the bounds in \eqref{eq:event-k}.

\begin{lem}[Bounds on the maximum number of steps $L$]~\\
Let us consider the variable $L$ defined in \eqref{eq:last-step}, we have:  
$$
\log_2\left(c_8\left(\frac{1}{\log\left(\frac{N}{\delta}\right)}\right)^{\frac{d+2}{1+d}}N^{\frac{d+3}{2+2d}}\right) \leq L 
$$
and
$$
L \leq \min\left(1+\log_2\left(\left(\frac{1}{c_6\log\left(\frac{N}{\delta}\right)}\right)^{(2+d)/(1+d)}N^{(3+d)/(2+2d)}\right), \log_2\left(\sqrt{N}\right)\right),
$$
where $c_8$, $c_6$ are the constants respectively defined in \eqref{eq:const8}, and \eqref{eq:const6}.
\end{lem}

\begin{proof}~\\
By definition of $L$, we have 
$$N\leq \sum_{i=0}^{L+1}N_i\Pi(A_i)$$

and we have as in the proof of Lemma \ref{lem:excess-error}
$$N_{L}\geq c_8\left(\frac{1}{\log\left(\frac{N}{\delta}\right)}\right)^{(d+2)/(1+d)}N^{(d+2)/(1+d)}.$$

Besides, as $N_L=2^LN_0$ and $N_0=\sqrt{N}$, we obtain the first inequality

\begin{equation}
\label{eq:depth-inf}
L\geq \log_2\left(c_8\left(\frac{1}{\log\left(\frac{N}{\delta}\right)}\right)^{(d+2)/(1+d)}N^{(d+3)/(2+2d)}\right)
\end{equation}
We can get the second inequality by starting with \eqref{eq:last-step}, that is:
$$
N_L\Pi(A_L)\leq N.
$$

Furthermore, as $\Pi(A_L)=\varepsilon_L=\min\left(1,c_6\log\left(\frac{N}{\delta}\right) N_{L-1}^{-1/(2+d)}\right)$ (see \eqref{eq:rejecteq}), we get 
$$N_L\min\left(1,c_6\log\left(\frac{N}{\delta}\right) N_{L-1}^{-1/(2+d)}\right)\leq N.$$
If $1\leq c_6\log\left(\frac{N}{\delta}\right) N_{L-1}^{-1/(2+d)}$, then

\begin{equation}
\label{eq:depth-sup1}
    L\leq \log_2\left(\sqrt{N}\right)
\end{equation}
On the other hand, if $1> c_6\log\left(\frac{N}{\delta}\right) N_{L-1}^{-1/(2+d)}$ then 

\begin{equation}
\label{eq:depth-sup2}
L\leq 1+\log_2\left(\left(\frac{1}{c_6\log\left(\frac{N}{\delta}\right)}\right)^{(2+d)/(1+d)}N^{(3+d)/(2+2d)}\right)\end{equation}
Finally, by combining  \eqref{eq:depth-inf}, \eqref{eq:depth-sup1}, and \eqref{eq:depth-sup2}, we get the second inequality. 
 
\end{proof}

\section{Proof of Theorem \ref{theo-rate}}
We firstly prove that in the event $E$, the classifier $g_{\hat{\eta}_k}$ does not make any error of classification in the set $A_k\setminus A_{k+1}$ for all $k=0,\ldots,L-1$, where $L$ is defined by \eqref{eq:last-step}.
\begin{lem}[Correct classification]~\\
\label{lem:correct-class}
Let $E$ be the event defined by \eqref{eq:event}. Under Assumption~\ref{ass:boundedness}, the Bayes classifier $g^*$ agrees with $g_{\hat{\eta}_k}$ on the set $A_{k}\setminus A_{k+1}$ for $k\in\lbrace 0,\ldots,L-1\rbrace$, where $L$ is defined by \eqref{eq:last-step}, and $\hat{\eta}_k$ by \eqref{eq:estimator}.  
\end{lem}

\begin{proof}~\\
Let us start by stating general facts that hold for a generic estimator $\hat{\eta}$ and the corresponding score function $\hat{f}(x) = \max(\hat{\eta}(x), 1-\hat{\eta}(x))$.
We consider $F_{f}$, and $F_{\hat{f}}$ the cumulative distribution of $f(X)$ and $\hat{f}(X)$, where $f(x) = \max(\eta(x), 1-\eta(x))$.
Let $t \in (1/2,1)$, we have that conditional on the data
\begin{equation*}
F_{\hat{f}}(t) \leq \left|F_{\hat{f}}(t)-F_{f}(t)\right| + F_{f}(t).    
\end{equation*}
Besides, the following relation holds:
\begin{equation*}
\left|F_{\hat{f}}(t)-F_{f}(t)\right|  \leq \mathbb{E}_X\left[\one_{\{\|\hat{f}-f\|_{\infty} \geq |f(X)-t|\}}\right] \leq 2C\|\hat{f}-f\|_{\infty},
\end{equation*}
where $C$ is the bound on the density $f$ provided in Assumption~\ref{ass:boundedness}. Using again 
Assumption~\ref{ass:boundedness} we can write
\begin{equation*}
F_{f}(t) \leq C\left(t-\frac{1}{2}\right).
\end{equation*}
We then deduce that for all $t \in (1/2,1)$, conditional on the data
\begin{eqnarray}
\label{eq:borneFhatf}
F_{\hat{f}}(t)  \leq  2C\|\hat{f}-f\|_{\infty} + C\left(t-\frac{1}{2}\right) 
            \leq 
            2C\|\hat{\eta}-\eta\|_{\infty} + C\left(t-\frac{1}{2}\right).
\end{eqnarray}
Given iteration $k\in \{ 0,\ldots, L-1\}$, we set $\hat{t}_k = \|\hat{\eta}_k-\eta\|_{\infty,A_{k}}$, and $t_k = \frac{1}{2}+ \hat{t}_k$. Thanks to~\eqref{eq:borneFhatf}, with $\hat{\eta} =\hat{\eta}_k $ and $t = t_k$, we deduce that (conditional on $A_k$)
\begin{equation*}
F_{\hat{f}_k}(t_k) \leq 3C\hat{t}_k.
\end{equation*}
Then, in the event $E$, we have that 

\begin{equation}
\label{eq:const6}
F_{\hat{f}_k}(t_k)\leq  c_6\log\left(\frac{N}{\delta}\right) N_k^{-1/(2+d)},
\end{equation}
where $c_6=3c_5C$, and $c_5$ is defined in \eqref{eq:union-bound-final}. Hence, 
\begin{equation}
\label{eq:rejecteq}
F_{\hat{f}_k}(t_k)\leq \min\left(1,c_6\log\left(\frac{N}{\delta}\right) N_k^{-1/(2+d)}\right) \leq \varepsilon_{k+1}
\end{equation}
This implies that $\lambda_{k+1} \geq \frac{1}{2} + \hat{t}_k$ by the definition of of $\lambda_{k+1}$.

Let $x$ $\in$ $A_k\setminus A_{k+1} = \{x \in A_k, \;\; \hat{f}_k(x) > \lambda_{k+1}\}$. Necessarily,  we have $$\hat{f}_k(x)-\frac{1}{2}>\|\hat{\eta}_k-\eta\|_{\infty,A_{k}}\geq \vert\hat{\eta}_k(x)-\eta(x)\vert$$
which implies $g_{\eta}(x)=g_{\hat{\eta}_k}(x).$
\end{proof}

\begin{lem}[Excess-error]\label{lem:excess-error}~\\
Let $g_{\hat{\eta}}$ be the classifier provided by our algorithm, on the event $E$, we have 
$$R(g_{\hat{\eta}})-R(g_{\eta})\leq \tilde{O}\left(N^{-\frac{2}{d+1}}\right),$$
where $\tilde{O}$ hides some constants and logarithmic factors.
\end{lem}

\begin{proof}
Let us consider the sequence $(A_k)_{0\leq k\leq L}$ used in our algorithm.  It is not difficult to see that $\{A_k\setminus A_{k+1}, k=0, \ldots, L-1\}\cup A_L$ forms a partition of $[0,1]^d$, where $L$ is defined by \eqref{eq:last-step}.

In this case, the excess-risk of $g_{\hat{\eta}}$ can be rewritten  as: 

$$R(g_{\hat{\eta}})-R(g^*)=\sum_{j=0}^{L-1}\int_{\{g_{\hat{\eta}}\neq g^*\}\cap \{A_j\setminus A_{j+1}\}}\vert 2\eta(x)-1\vert d\Pi(x)+ \int_{A_L\cap \{g_{\hat{\eta}}\neq g^*\}} \vert 2\eta(x)-1\vert d\Pi(x) $$

and thus
\begin{equation}
\label{eq:excess-decomposition}
R(g_{\hat{\eta}})-R(g^*)=2 \sum_{j=1}^{L-1} \E_X\left[|\eta(X)-\frac{1}{2}| \one_{\{g^*(X)\neq g_{\hat{\eta}_j}(X)\}} \one_{\{A_j \setminus A_{j+1}\}} \right] +2 \E_X\left[|\eta(X)-\frac{1}{2}| \one_{\{g^*(X)\neq g_{\hat{\eta}_L}(X)\}} \one_{\{A_L\}}\right] 
\end{equation}
Due to the Lemma \ref{lem:correct-class}, the first term in the r.h.s of \eqref{eq:excess-decomposition} is zero in the event $E$.
Thus we get 
\begin{align*}
R(\hat g)-R(g^*)& =2 \E_X\left[|\eta(X)-\frac{1}{2}| \one_{\{g^*(X)\neq g_{\hat{\eta}_L}(X)\}} \one_{\{A_L\}}\right]\\
                &\le 2 \E_X\left[|\eta(X)-\frac{1}{2}| \one_{|\hat{\eta}(X)-\frac{1}{2}|< |\hat{\eta}_L(X)-\eta(X)|} \one_{\{A_L\}}\right]  
\end{align*}
We thus have 
\begin{align}
R(\hat g)-R(g^*) &\le 2 \|\hat{\eta}_L-\eta\|_{\infty,A_L}.
\E_X\left[\one_{|\hat{\eta}(X)-\frac{1}{2}|< |\hat{\eta}_L(X)-\eta(X)|}\right]\nonumber\\
                 &\le 4C\|\hat{\eta}_L-\eta\|_{\infty,A_L}^{2}\label{eq:pre-rate}\;\text{by Assumption \ref{ass:boundedness}}.
\end{align}
By Lemma \ref{lem:event}, we get with probability at least $1-\delta$
\begin{equation}
\label{eq:final-risk}
R(\hat g)-R(g^*) \le 4Cc_6 \log^2\left(\frac{N}{\delta}\right) N_L^{-2/(2+d)}.
\end{equation}

Besides, because of the geometric progression of $N_{j}$, and the definition of $L$,  we have
\begin{align*}
N &\leq\sum_{j=0}^{L+1} N_{j}\Pi(A_j)\\
  &= \sum_{j=0}^{L+1} N_{j}\varepsilon_{j}\\
  &\le N_0+c_6\log\left(\frac{N}{\delta}\right)\sum_{j=1}^{L+1} N_j N_{j-1}^{-1/(2+d)}\\
  &=N_0+2c_6\log\left(\frac{N}{\delta}\right)\sum_{j=1}^{L+1} N_{j-1}^{(d+1)/(2+d)}\\
  &\leq N_0+ c_7\log\left(\frac{N}{\delta}\right) N_{L+1}^{(d+1)/(2+d)}\quad \text{for some constant c7}.
\end{align*} 
Thus we get 
\begin{align}
    N-N_0\leq c_7\log\left(\frac{N}{\delta}\right) N_{L+1}^{(d+1)/(2+d)}&\Longrightarrow \dfrac{1}{4}N\leq c_7\log\left(\frac{N}{\delta}\right) N_{L+1}^{(d+1)/(2+d)} \quad\text{as}\; N_0=\sqrt{N}\leq \dfrac{3}{4}N \nonumber\\
                        &\Longrightarrow N_{L}\geq c_8\left(\frac{1}{\log\left(\frac{N}{\delta}\right)}\right)^{(d+2)/(1+d)}N^{(d+2)/(1+d)}\label{eq:lowerboundsNL},
\end{align}
where 
\begin{equation}
\label{eq:const8}
c_8=\dfrac{1}{2} \left(\dfrac{1}{4c_7}\right)^{(d+2)/(1+d)}.
\end{equation}
Thus, \eqref{eq:final-risk} becomes 

$$R(g_{\hat{\eta}})-R(g_{\eta})\leq \tilde{O}\left(N^{-\frac{2}{d+1}}\right).$$

\end{proof}

\end{document}